\documentclass[final]{siamart220329}
\usepackage{amssymb, amsfonts}
\usepackage{mathabx}
\usepackage[mathscr]{euscript}
\usepackage{algorithmic}
\usepackage{mathtools}
\usepackage{bm}
\usepackage{booktabs}
\usepackage{array}
\usepackage[caption = false]{subfig}
\usepackage{graphicx,epstopdf}
\usepackage{physics}
\usepackage{chemformula}
\usepackage{multirow}

\numberwithin{equation}{section}

\newtheorem{remark}{Remark}
\newtheorem{assumption}{Assumption}

\newcommand{\bbR}{\mathbb{R}}
\newcommand{\bbC}{\mathbb{C}}
\newcommand{\calN}{\mathcal{N}}
\renewcommand{\tr}{\mathrm{tr}}
\newcommand{\diag}{\mathrm{diag}}
\newcommand{\fmu}{f_{\mu,W}}
\newcommand{\hess}{\mathrm{Hess}\,}

\newcommand{\ie}{i.e.}
\newcommand{\eg}{e.g.}

\newcommand{\norb}{{N_{\mathrm{orb}}}}
\newcommand{\nelec}{{N_{\mathrm{e}}}}

\newcommand{\WTPM}{WTPM}
\newcommand{\fnorm}[1]{\norm{#1}_{\mathrm{F}}}
\newcommand{\conj}{{\mathrm{H}}}
\newcommand{\argmax}{\mathop{\arg\max}}
\newcommand{\argmin}{\mathop{\arg\min}}

\title{Weighted trace-penalty minimization for full configuration
interaction \thanks{Submitted to the editors Jan. 16, 2023. \funding{W.
Gao was partially supported by National Key R\&D Program of China under
Grant No. 2020YFA0711900, 2020YFA0711902 and National Natural Science
Foundation of China under Grant No. 71991471. Y. Li and H. Shen were
partially supported by National Natural Science Foundation of China under
Grant No. 12271109 and Science and Technology Commission of Shanghai
Municipality under Grant No. 22TQ017.}}}

\author{
    Weiguo Gao
    \thanks{School of Mathematical Sciences and School of
    Data Science, Fudan University, Shanghai 200433, China; Shanghai
    Artificial Intelligence Laboratory, Shanghai 200232, China
    (\email{wggao@fudan.edu.cn}).}
    \and
    Yingzhou Li
    \thanks{School of Mathematical Sciences, Fudan University, Shanghai
    200433, China (\email{yingzhouli@fudan.edu.cn}).}
    \and
    Hanxiang Shen
    \thanks{Shanghai Center for Mathematical Science, Fudan University,
    Shanghai 200438, China (\email{hxshen19@fudan.edu.cn}).}
}

\headers{Weighted trace-penalty minimization}{Weiguo Gao, Yingzhou
Li, and Hanxiang Shen}

\begin{document}

\maketitle

\begin{abstract}
    A novel unconstrained optimization model named weighted trace-penalty
    minimization (\WTPM{}) is proposed to address the extreme eigenvalue
    problem arising from the Full Configuration Interaction (FCI) method.
    Theoretical analysis shows that the global minimizers of the \WTPM{} objective
    function are the desired eigenvectors, rather than the eigenspace.
    Analyzing the condition number of the Hessian operator in detail
    contributes to the determination of a near-optimal weight matrix.
    With the sparse feature of FCI matrices in mind, the coordinate
    descent (CD) method is adapted to \WTPM{} and results in \WTPM{}-CD
    method.  The reduction of computational and storage costs in each
    iteration shows the efficiency of the proposed algorithm.  Finally,
    the numerical experiments demonstrate the capability to address
    large-scale FCI matrices.
\end{abstract}

\begin{keywords}
Eigensolver, Weighted Trace-Penalty Minimization, Full Configuration
Interaction, Coordinate Descent Method
\end{keywords}

\begin{MSCcodes}
    65F15
\end{MSCcodes}

\section{Introduction}

The time-independent, non-relativistic Schr{\"o}dinger equation is a
linear Hermitian extreme eigenvalue problem, 
\begin{equation} \label{eq:eigenproblem}
    H \ket{\psi_i} = E_i\ket{\psi_i}, \quad i=1,2,\dots,p, 
\end{equation}
where $H$ is a Hamiltonian operator, $(E_i,\ket{\psi_i})$ denotes the
ground-state and lowest few excited-state energies and their corresponding
wavefunctions, and $p$ is the number of desired eigenvalues. Efficiently
solving the Schr{\"o}dinger equation plays a fundamental role in the field
of electronic structure calculation. The problem is
of high-dimensionality, \ie, $\ket{\psi} = \psi(x_1, \dots, x_\nelec)$ for
$x_i \in \bbR^3$ being the position of electrons and $\nelec$ being the
number of active electrons in the system. Further, identical electron
wavefunctions admit an antisymmetry property, which corresponds to the
Pauli exclusion principle. Full configuration interaction~(FCI) is a
variational method that solves~\eqref{eq:eigenproblem} numerically exactly
within the space of all Slater determinants~\cite{knowles1984new,
knowles1989determinant, sherrill1999configuration}. By the nature of
Slater determinants, the antisymmetry property is encoded into the
many-body basis functions. Nevertheless, the high-dimensionality feature
still leads to extremely large Hamiltonian matrices. More precisely, the
size of Hamiltonian matrices grows factorially with respect to the number
of electrons and orbitals in the system. For a system with $\norb$
orbitals and $\nelec$ electrons, the number of Slater determinants is
$O\Big(\tbinom{\norb}{\nelec} \Big)$~\cite{knowles1989unlimited}.
A single $\ch{H2O}$ molecule with cc-pVDZ basis~($24$ orbitals) and 10
active electrons leads to a $4.53 \cdot 10^8 \times 4.53 \cdot 10^8$
Hamiltonian matrix~\cite{wangCoordinateDescentFull2019a}. The antisymmetry
property of the problems leads to the notorious sign problem.

In electronic structure calculation, widely used eigensolvers can be
classified into two groups: Krylov subspace methods and optimization
methods. Krylov subspace methods include Chebyshev-Davidson
algorithm~\cite{zhouChebyshevDavidsonAlgorithm2007}, locally optimal block
preconditioned conjugate gradient
method~(LOBPCG)~\cite{knyazev2001toward}, block Krylov-Schur
algorithm~\cite{zhou2008block}, projected preconditioned conjugate
gradient algorithm
(PPCG)~\cite{vecharynskiProjectedPreconditionedConjugate2015}, etc. In all
these methods, an explicit orthogonalization step is carried out every few
iterations, which costs a significant amount of computational resource
throughout the algorithm. The other group, optimization methods,
transforms the eigenvalue problem into an optimization problem with or
without the orthogonality constraint. The orthogonality constrained
optimization problem is also known as the Stiefel manifold optimization.
The corresponding optimization
algorithms~\cite{absilOptimizationAlgorithmsMatrix2008,
gaoNewFirstOrderAlgorithmic2018, huBriefIntroductionManifold2020,
sameh2000trace} require either a projection or a retraction step to keep
the iteration on the manifold. The computational costs for these
projections and retractions remain the same as that of the
orthogonalization. Solving the eigenvalue problem via an unconstrained
optimization is popular recently, especially when the large-scale
eigenvalue problems are considered. Such methods include symmetric
low-rank product model~\cite{liCoordinateWiseDescentMethods2019,
liuEfficientGaussNewton2015, wangCoordinateDescentFull2019a}, orbital
minimization method~(OMM)~\cite{corsettiOrbitalMinimizationMethod2014,
luOrbitalMinimizationMethod2017}, trace-penalty
minimization~\cite{wen2016trace}, etc. All these methods converge to the
invariant subspace corresponding to the smallest eigenvalues, and require
a single Rayleigh-Ritz procedure to extract eigenvectors from the
invariant subspace.

However, due to the large-scale matrix
size, none of the eigensolvers mentioned above can be applied directly to
the FCI eigenvalue problem. Almost all of them run into the memory
bottleneck. Many non-standard eigensolvers are designed particularly for
FCI ground-state computation. Density matrix renormalization
group (DMRG)~\cite{chan2011thedensity, olivares2015theabinitio,
white1999ab} approximates the high-dimensional wavefunctions by a tensor
train. The underlying eigenvalue problem is solved by the vanilla power
method. FCI quantum Monte Carlo
(FCIQMC)~\cite{boothExactDescriptionElectronic2013,
boothFermionMonteCarlo2009, luFullConfigurationInteraction2020} adopts the
quantum Monte Carlo method to overcome the sign problem and the curse of
dimensionality. Related
methods~\cite{clelandCommunicationsSurvivalFittest2010,
petruzieloSemistochasticProjectorMonte2012} in this family further adjust
the variance-bias trade-off to obtain efficient algorithms.  Selected
configuration
interaction~(SCI)~\cite{holmesHeatBathConfigurationInteraction2016,
iterative1973huron, schriberAdaptiveConfigurationInteraction2017,
tubmanDeterministicAlternativeFull2016} employs perturbation analysis of
eigenvalue problems and selects important configurations accordingly. Then
the eigenvalue problem of the selected principal submatrix is solved via
traditional eigensolvers. Various methods in this family differ from each
other in the computationally efficient approximations to the perturbation
result. Coordinate descent
FCI~(CDFCI)~\cite{liCoordinateWiseDescentMethods2019,
wangCoordinateDescentFull2019a} applies the coordinate descent method on
the symmetric low-rank product model to select important configurations
and to update the coefficients. A carefully designed compression scheme is
incorporated to overcome the memory bottleneck.

Furthermore, the non-standard eigensolvers mentioned above can be adapted
to the low-lying excited-states computation.
DMRG~\cite{Baiardi_2017_VibrationalDensityMatrix} and
FCIQMC~\cite{bluntExcitedstateApproachFull2015} adopt the deflation idea
and compute excited-states one-by-one.
SCI~\cite{schriberAdaptiveConfigurationInteraction2017,
tubmanDeterministicAlternativeFull2016} needs to select many more
configurations to capture the important configurations for excited states.
CDFCI~\cite{Li_2022_CDFCIforExcitedStates} could be naturally extended and
converge to the invariant subspace formed by the ground-state and
excited-states. The extra rotation matrix within the invariant subspace
leads to less sparse iteration variables and increases the memory cost.
Specifically targeting the low-lying excited-states computation, the
triangularized orthogonalization free
method~(TriOFM)~\cite{gaoTriangularizedOrthogonalizationfreeMethod2020,
gaoGlobalConvergenceTriangularized2021} proposes a triangularized
iterative scheme converging to eigenvectors directly without any
projection or Rayleigh-Ritz step.

In this paper, inspired by the weighted subspace-search variational quantum
eigensolver~\cite{nakanishi2019subspace} from quantum computing and the
trace-penalty model~\cite{wen2016trace}, we propose an unconstrained
optimization model called weighted trace-penalty minimization (\WTPM{}),
\begin{equation} \label{eq:WeightedTracePenalty}
    \min_{X \in \bbR^{n\times p}} \fmu(X) = \frac{1}{2} \tr(X^\top AX)
    + \frac{\mu}{4}\fnorm{X^\top X-W}^2,
\end{equation}
where $A \in \bbR^{n \times n}$ is a symmetric matrix, $\mu > 0$ is the
penalty parameter and $W$ is a diagonal weight matrix with distinct
diagonal entries. 
Our analysis shows that
\eqref{eq:WeightedTracePenalty} does not have spurious local minima, and
the $2^p$ isolated global minima of \eqref{eq:WeightedTracePenalty} are
scaled eigenvectors corresponding to the smallest $p$ eigenvalues of $A$.
Moreover, we calculate the condition number of the $W$-dependent Hessian
matrix of $\fmu(X)$ at global minima, which leads to the local convergence
rate for the first-order method.  A near-optimal weight matrix $W$ is,
then, derived to achieve a fast local convergence rate.

With the size of FCI problem size in mind, we focus on first-order methods
to address \eqref{eq:WeightedTracePenalty}. One choice is the gradient
descent (GD) method with Barzilai-Borwein (BB)
stepsize~\cite{barzilai1988two}. A more desirable choice is the coordinate
descent (CD) method, which reveals the sparsity in FCI problems
efficiently~\cite{wangCoordinateDescentFull2019a}. Hence, we tailor the CD
method for \eqref{eq:WeightedTracePenalty} and obtain an efficient
eigensolver for the FCI eigenvalue problem. Global convergence of the
proposed eigensolver can be proved in the same way as that of
CDFCI~\cite{liCoordinateWiseDescentMethods2019}, whereas the local linear
convergence is obtained with a rate related to the Hessian operator. We
emphasize that both GD and CD methods for \eqref{eq:WeightedTracePenalty}
converge to the scaled eigenvectors directly. Hence the expensive and
parallel inefficient Rayleigh-Ritz step is omitted entirely from both
methods.

Numerically, we test and compare the performance of the original
trace-penalty model~\cite{wen2016trace} and our \WTPM{} on FCI matrices
from practice. The numerical results show that both models with
first-order methods converge to desired global minima. Adding the extra
weight matrix in \WTPM{} does not destroy the efficiency of the original
trace-penalty model. For the FCI matrices, \WTPM{}-CD method converges in
far less cost of flops. \WTPM{}-CD finds the sparse representations of the
eigenvectors corresponding to the smallest few eigenvalues, whereas the
trace-penalty model requires an extra Rayleigh-Ritz step.

The rest of this paper is organized as follows. In
Section~\ref{sec:analysis}, we give a theoretical analysis on
\eqref{eq:WeightedTracePenalty} with a focus on the global minima and the
condition number of the Hessian operator. Section~\ref{sec:algorithm}
proposes algorithms for \eqref{eq:WeightedTracePenalty} and analyzes their
performance and complexity. Section~\ref{sec:numericalresult} reports the
numerical results showing the efficiency of our method. Finally, we
conclude this paper in Section~\ref{sec:conclusion} with some discussion
on future work.

\section{Model Analysis}
\label{sec:analysis}

This section focuses on the analysis of the energy landscape of the
weighted trace-penalty minimization model. We will first analyze the
stationary points and Hessian operator of \eqref{eq:WeightedTracePenalty} in
Section~\ref{sec:stationarypoints} and Section~\ref{sec:hessianmatrices}
respectively. Based on the condition number of the Hessian operator at the
global minimum, a near-optimal choice of the weight matrix $W$ is
discussed in Section~\ref{sec:optimalW} to achieve a near-optimal local
convergence rate for first-order methods. Finally in
Section~\ref{sec:hermitianmatrices}, we discuss the extensions of all
analysis results to Hermitian matrices.

We consider a real symmetric matrix $A$ of size $n\times n$. The
eigenvalue decomposition of $A$ is denoted as,
\begin{equation} \label{eq:ConditionOfA}
    A = V \Lambda V^\top,
\end{equation}
where $V = (v_1, v_2, \dots, v_n) \in \bbR^{n\times n}$ and $\Lambda =
\diag(\lambda_1, \lambda_2, \dots, \lambda_n)$ such that
\begin{equation} \label{eq:SpectrumOfA}
\lambda_1 < \lambda_2 < \cdots < \lambda_p
< \lambda_{p+1} \leqslant \cdots \leqslant \lambda_n.
\end{equation}
A pair $(\lambda_i, v_i)$ is an eigenpair of $A$. Throughout this paper,
we aim to compute the smallest $p$ eigenpairs of $A$ via \WTPM{}. The real
weight matrix in \eqref{eq:WeightedTracePenalty} satisfies the following
assumption.
\begin{assumption}
    \label{assump:ConditionOfW}
    The weight matrix is diagonal, $W = \diag(w_1, w_2,
    \dots, w_p)$ such that
    \begin{equation} \label{eq:ConditionOfW}
        w_1 > w_2 > \cdots > w_p > \frac{\lambda_p}{\mu}.
    \end{equation}
\end{assumption}

\subsection{Stationary Points}
\label{sec:stationarypoints}

Stationary points of \eqref{eq:WeightedTracePenalty} satisfy the
first-order necessary optimal condition,
\begin{equation} \label{eq:FirstOrderCondition}
    \nabla \fmu(X) = AX + \mu X(X^\top X-W) = 0.
\end{equation}
Left multiplying both sides of \eqref{eq:FirstOrderCondition} by the
transpose of $X$, we obtain, 
\begin{equation} \label{eq:DiagonalXX}
    X^\top AX = \mu X^\top X(W-X^\top X),
\end{equation}
where the left-hand side is symmetric. The symmetry property
of~\eqref{eq:DiagonalXX} leads to the fact that $X^\top XW = WX^\top X$.
Since $W$ is a diagonal matrix and all entries are distinct as in
\eqref{eq:ConditionOfW}, the equation $X^\top XW = WX^\top X$ implies that
$X^\top X$ is diagonal, \ie, columns of stationary point $X$ are either
zero or mutually orthogonal. In Theorem~\ref{thm:StationaryPoint}, we give
the explicit form of the stationary points of
\eqref{eq:WeightedTracePenalty}, where each
column of $X$ is either an eigenvector of $A$ or the zero vector.

\begin{theorem} [Stationary Points]
    \label{thm:StationaryPoint}
    Assume $A$ and $W$ satisfy \eqref{eq:ConditionOfA} and
    \eqref{eq:ConditionOfW} respectively. Any stationary point
    $\widehat{X}$ of \eqref{eq:WeightedTracePenalty} has the form
    \begin{equation}
        \widehat{X} = \widehat{U}_p \widehat{S}_p,
    \end{equation}
    where $\widehat{U}_p = (\hat{u}_1, \hat{u}_2, \dots, \hat{u}_p)$ and
    $\widehat{S}_p = \diag(\hat{s}_1, \hat{s}_2, \dots, \hat{s}_p)$ such
    that 
    \begin{equation}\label{eq:PropertyStationaryPoint}
        \begin{split}
            A \hat{u}_i = \sigma_i \hat{u}_i,
            \quad \hat{u}_i^\top \hat{u}_j = \delta_{ij} \text{ and,} \\
            \hat{s}_i \in \left\{0, \sqrt{w_i - \frac{\sigma_i}{\mu}}
            \right\}.
        \end{split}
    \end{equation}
\end{theorem}

\begin{proof}
As we discussed above, $\widehat{X}^\top \widehat{X}$ is a diagonal
matrix. Columns of \eqref{eq:FirstOrderCondition}, then, can be
represented as
\begin{equation}
    \label{eq:ReformulateFirstOrderCondition}
    A \hat{x}_i = d_i \hat{x}_i,\quad i = 1, 2, \dots, p,
\end{equation}
where $\hat{x}_i$ is the $i$-th column of $\widehat{X}$ and $d_i = \mu
(w_i - \hat{x}^\top_i \hat{x}_i)$. Vector $\hat{x}_i$ is either a zero
vector or an eigenvector of $A$. When $\hat{x}_i$ is a zero vector, it can
be represented as the form in the theorem for $\hat{s}_i = 0$. When
$\hat{x}_i$ is not a zero vector, we denote $\hat{x}_i = \hat{u}_i
\hat{s}_i$ for $\hat{u}_i$ being a unit length eigenvector of $A$
associated with eigenvalue $\sigma_i$ and $\hat{s}_i$ being a positive
scalar. Then $\hat{s}_i$ satisfies
\begin{equation*}
    \sigma_i = d_i = \mu (w_i - \hat{s}_i^2) \Rightarrow
    \hat{s}_i = \sqrt{w_i - \frac{\sigma_i}{\mu}}.
\end{equation*}
Recall that $\widehat{X}^\top \widehat{X}$ is a diagonal matrix. For
nonzero $\hat{s}_i$ and $\hat{s}_j$, it is required that $\hat{u}_i^\top
\hat{u}_j = \delta_{ij}$. When $\hat{s}_i$ is zero, we could always find
an extra orthogonal eigenvector of $A$, such that $\hat{u}_i^\top
\hat{u}_j = \delta_{ij}$ still holds. Therefore, we proved that the
stationary points of \eqref{eq:WeightedTracePenalty} must admit the form
in the theorem and any point that admits the form
in the theorem is a stationary point.
\end{proof}
 
Furthermore, we will distinguish the local and global minima from saddle
points. Interestingly, it can be shown that the columns of global
minimizer $X^*$ are eigenvectors of $A$ associated with the $p$ smallest
eigenvalues.

\begin{theorem} [Global Minima]
    \label{thm:GlobalMin}
    Assume $A$ and $W$ satisfy \eqref{eq:ConditionOfA} and
    \eqref{eq:ConditionOfW}. Any global
    minimizer $X^*$ of \eqref{eq:WeightedTracePenalty} has the form
    \begin{equation}\label{eq:GlobalMin}
        X^* = V_p S_p,
    \end{equation}
    where $V_p = (v_1,v_2,\dots,v_p)$ and $S_p = \diag(\pm s_1,\pm
    s_2,\dots,\pm s_p)$ such that
    \begin{equation*}
        s_i = \sqrt{w_i - \frac{\lambda_i}{\mu}} . 
    \end{equation*}
\end{theorem}

\begin{proof}

By Theorem~\ref{thm:StationaryPoint}, the stationary point has the form
$\widehat{X} = \widehat{U}_p \widehat{S}_p$. Substituting it into the
objective function leads to,
\begin{equation*}
    2 \fmu(\widehat{X}) = \sum_{i=1}^p \bigl[\hat{s}_i^2 \sigma_i
    + \frac{\mu}{2} (\hat{s}_i^2 - w_i)^2\bigr]
    = \sum_{i=1}^p \frac{\mu}{2}w_i^2 - \sum_{i \in \{\hat{s}_i \neq 0\}}
    \frac{(\sigma_i - \mu w_i)^2}{2\mu},
\end{equation*}
where $\sigma_i$ is one of the eigenvalues of $A$ and the second equality
is due to the expression of $\hat{s}_i$. If some $\hat{s}_i
= 0$, then we have $\sigma_i - \mu w_i = 0$. Under
Assumption~\ref{assump:ConditionOfW}, there are at least $p$ eigenvalues
to make $-\tfrac{(\lambda - \mu w_i)^2}{2\mu} < 0$ and at least one of
them is not in $\{\sigma_i\}_{i=1}^p$. Replacing $\sigma_i$ by one of the
unselected eigenvalue with $-\tfrac{(\lambda - \mu w_i)^2}{2\mu} < 0$
would lead to a smaller objective function value. Hence if $\hat{s}_i = 0$
for some $i$, the stationary point is not a global minimizer. Notice that
if $\hat{s}_i \neq 0$ then $\sigma_i < \mu w_i$. We only need to show that
$\sigma_i = \lambda_i$.

First, we claim that $\sigma_1,\sigma_2,\dotsc,\sigma_p$ must be a
permutation of $\lambda_1,\lambda_2,\dotsc,\lambda_p$. If not, for example
there exists $\sigma_{i_0}$ such that $\mu w_{i_0} > \sigma_{i_0} >
\lambda_p$. There also exists $\lambda_j$ not used such that
$\lambda_j \leqslant
\lambda_p$. If $\sigma_{i_0} = \lambda_j$ instead, the value of
$\fmu(\widehat{X})$ will decrease. This contradicts with the minimal
property so our claim holds.

Next, since there are at least $p$ eigenvalues smaller than any $\mu w_i$
due to Assumption~\ref{assump:ConditionOfW}, the
optimization~\eqref{eq:WeightedTracePenalty} changes into the form as
\begin{equation}\label{eq:StationaryPointFuncVal}
    \max_{\substack{\sigma_i \in \{\lambda_1,\dotsc,\lambda_p\}\\
    \sigma_i \neq \sigma_j\;\text{if }i\neq j}}
    \sum_{i=1}^p (\sigma_i-\mu w_i)^2
    =\sum_{i=1}^p -2\mu\sigma_i w_i + \text{const}.
\end{equation}
According to Rearrangement Inequality~\cite{Hardy_2001_Inequalities}, we
can conclude that the minimum value is reached if $\sigma_i = \lambda_i$,
\ie 
\[ X^* = V_p S_p.\]
\end{proof}

Theorem~\ref{thm:GlobalMin} presents an interesting result that any global
minimizer of the \WTPM~\eqref{eq:WeightedTracePenalty}
is composed of desired eigenvectors directly rather
than the underlying invariant subspace. 

Another result is that there are no spurious local minima, \ie, the rest
stationary points are all strict saddle points. To show this, we first
introduce the Hessian operator of $\fmu$,
\begin{equation} \label{eq:hessian}
    \hess \fmu(X)[Z] = A Z + \mu (Z X^\top X + X Z^\top X + X X^\top Z
    - Z W),
\end{equation}
where $Z \in \bbR^{n \times p}$ is an arbitrary matrix. A stationary point
$\hat{X}$ is called a strict saddle point of $\fmu$ if and only if the
Hessian operator $\hess \fmu(\hat{X})$ has negative
eigenvalues.~\footnote{Here the strict saddle point includes maximizer of
the problem, \ie, a stationary point with negative semi-definite Hessian
operator.} 

In Theorem~\ref{thm:NoLocalMinima}, we prove that
\eqref{eq:WeightedTracePenalty} does not have any spurious local minimum.

\begin{theorem}
    \label{thm:NoLocalMinima}
    Assume $A$ and $W$ satisfy \eqref{eq:ConditionOfA} and
    \eqref{eq:ConditionOfW}. There are no local minima other than the
    global minimizers $X^*$.
\end{theorem}

\begin{proof}

First, we discuss the case that the stationary point $\widehat{X} =
\widehat{U}_p\widehat{S}_p$ is of full column rank, \ie, $\hat{s}_i \neq
0, \forall i$. We focus on the sign of 
\begin{multline} \label{eq:TraceHessian}
    \tr\left(Z^\top \hess \fmu (\hat{X})[Z]\right) =\\
    \tr\left(Z^\top (A + \mu\hat{X} \hat{X}^\top ) Z
    + \mu Z^\top Z (\hat{X}^\top \hat{X} - W)
    + \mu Z^\top \hat{X} Z^\top \hat{X}\right),
\end{multline}
for $Z \in \bbR^{n\times p}$. If \eqref{eq:TraceHessian} at $\hat{X}$ is
strictly negative for a particular $Z$, then we could conclude that the
Hessian operator has negative eigenvalues and hence $\hat{X}$ is a strict
saddle point.

Since $\hat{X}$ is not global minimizer, there exists an eigenvector $v_i$
in $\{v_1,v_2,\dotsc,v_p\}$ satisfying $v_i^\top  \hat{U}_p = 0$ and
$\exists j,\;\lambda_i < \sigma_j$. Let $Z$ be the matrix whose $j$-th
column is $v_i$ and others are zero. Then we have
\begin{equation}\label{eq:FullRankCounterEx2}
    \tr\left(Z^\top  \hess \fmu (\hat{X})[Z]\right) = \lambda_i - \sigma_j < 0, 
\end{equation}
implying that full-rank stationary points are all saddle points except
global minima.

Next, we discuss the rank-deficient case. Without loss of generality we
assume the $j$-th column of $\hat{X}$ is zero, and there also exists
an eigenvector $v_i$ in $\{v_1,v_2,\dotsc,v_p\}$ satisfying $v_i^\top
\hat{X} = 0$. Let $j$-th column of $Z$ be $v_i$ and others are zero. We
can obtain
\begin{equation}\label{eq:RankDeficientCounterEx1}
    \tr\left(Z^\top  \hess \fmu (\hat{X})[Z]\right) = 
    \lambda_i - \mu w_j < 0. 
\end{equation} 
Hence, rank-deficient stationary points are all saddle points. 
\end{proof}

\subsection{Hessian Operator}
\label{sec:hessianmatrices}

For first-order optimization methods, the local convergence rate relies on
the condition number of the Hessian operator~\cite{boyd_convex_2004}.
Specifically, in the neighborhood of a global minimum $X^*$, the error of
a gradient descent method with exact line search decreases linearly as
\begin{equation} \label{eq:ConvergeRateOfGradient}
    \fmu(X^{(j+1)}) - \fmu(X^*)  \leqslant 
    \bigl( 1 - \kappa^{-1} \bigr)
    \bigl( \fmu(X^{(j)}) - \fmu(X^*) \bigr)
\end{equation}
where $X^{(j)}$ denotes the iteration variable at $j$-th iteration and
$\kappa$ denotes the condition number of the Hessian operator at $X^*$.

From Theorem~\ref{thm:GlobalMin}, we find that all global minima of
\eqref{eq:WeightedTracePenalty} are isolated points and the Hessian
operator at any global minimizer $X^*$ is strictly positive definite. In
order to give a depiction of the local convergence rate of
\eqref{eq:WeightedTracePenalty}, we present a tight estimation of the
condition number of the Hessian operator in
Theorem~\ref{thm:ConditionNumber}.

\begin{theorem}
    \label{thm:ConditionNumber}
    Assume $A$ and $W$ satisfy \eqref{eq:ConditionOfA} and
    \eqref{eq:ConditionOfW}. Let $X^*$ be a global minimizer of $\fmu$.
    Then
    \begin{subequations} \label{eq:ConditionNumber}
        \begin{align}
            \kappa\left(\hess \fmu (X^*)\right) 
            & \triangleq \frac{
                \max \limits_{\fnorm{Z}=1} \tr\left(Z^\top
                \hess \fmu (X^*)[Z] \right)
                }
                {
                \min \limits_{\fnorm{Z}=1} \tr\left(Z^\top
                \hess \fmu (X^*)[Z]\right)
                },\label{eq:ConditionNumber:a}\\
            & = \frac{
                \max\left\{
                    \lambda_n-\lambda_1,
                    2(\mu w_1 - \lambda_1), 
                    \max\limits_{i<j}\lambda_{\max}(M_{ij})
                    \right\}
                    }
                    {
                \min\left\{
                    \lambda_{p+1}-\lambda_p,
                    2(\mu w_p - \lambda_p), 
                    \min\limits_{i<j}\lambda_{\min}(M_{ij})
                    \right\}
                    }\label{eq:ConditionNumber:b}, 
        \end{align}
    \end{subequations}
    where $\lambda_{\max}(\cdot)$ and $\lambda_{\min}(\cdot)$ denote the
    largest and the smallest eigenvalue of the matrix 
    \[
        M_{ij} = 
        \begin{pmatrix}
            \mu w_i - \lambda_j
            & \sqrt{(\mu w_i - \lambda_i)(\mu w_j - \lambda_j)}\\
            \sqrt{(\mu w_i - \lambda_i)(\mu w_j - \lambda_j)}
            & \mu w_j - \lambda_i
        \end{pmatrix}
    \]
    respectively.
\end{theorem}

\begin{proof}

Given a $Z \in \bbR^{n\times p}$ such that $\fnorm{Z}=1$, it can be
represented as 
\begin{align*}
    &Z = V_p Z_1 + \widebar{V}_p Z_2,\\
    &\tr(Z_1^\top Z_1)+\tr(Z_2^\top Z_2) = 1,
\end{align*}
where $\widebar{V}_p = (v_{p+1},\ldots,v_n)$, $Z_1 \in \bbR^{p\times p}$,
and $Z_2 \in \bbR^{(n-p) \times p}$. Using the expressions of the global
minimizer \eqref{eq:GlobalMin} and the Hessian
operator~\eqref{eq:hessian}, we obtain,
\begin{multline} \label{eq:traceZEq}
    \tr\left(Z^\top \hess \fmu (X^*)[Z]\right) 
        =\tr(Z_2^\top \widebar{\Lambda}_p Z_2) - \tr(Z_2^\top Z_2\Lambda_p)\\
        +\tr(\mu Z_1^\top W Z_1)-\tr(Z_1^\top Z_1\Lambda_p) 
        +\tr(\mu Z_1^\top S_pZ_1^\top S_p), 
\end{multline}
where $\Lambda_p = \diag(\lambda_1,\ldots,\lambda_p)$ and
$\widebar{\Lambda}_p = \diag(\lambda_{p+1},\ldots,\lambda_n)$. The first
two terms in \eqref{eq:traceZEq} can be bounded as,
\begin{equation}\label{eq:BoundsOfZ2Part}
    \tr(Z_2^\top Z_2)(\lambda_{p+1}-\lambda_p) \leqslant 
    \tr(Z_2^\top  \bar{\Lambda}_p Z_2) - \tr(Z_2^\top Z_2\Lambda_p) \leqslant 
    \tr(Z_2^\top Z_2)(\lambda_{n}-\lambda_1),
\end{equation}
where the upper and lower bounds are achieved when $Z_2 = c \cdot
(e_{n-p},0,\ldots,0)$ or $Z_2 = c \cdot (0,\ldots,0,e_{1})$, where $e_i$
denotes the $i$-th column of the identity matrix and $c$ is a
scalar.

Next, we would like to bound the terms in \eqref{eq:traceZEq} associated
with $Z_1$. Let $Z_1 = (z_{ij})_{p\times p}$. A short calculation shows
that
\begin{equation}\label{eq:Z1Part}
    \begin{split}
        & \tr(\mu Z_1^\top W Z_1) - \tr(Z_1^\top Z_1 \Lambda_p)
        + \mu \tr(Z_1^\top S_p Z_1^\top S_p) \\
        = & 2\sum_{i=1}^p (\mu w_i - \lambda_i) z_{ii}^2 + 
        \sum_{j>i}
        \begin{pmatrix}
            z_{ij} & z_{ji}
        \end{pmatrix}
        M_{ij} 
        \begin{pmatrix}
          z_{ij}\\z_{ji}
        \end{pmatrix},
    \end{split}
\end{equation}
where $M_{ij} \in \bbR^{2\times 2}$ as defined in the theorem. Eigenvalues
of $M_{ij}$ can be calculated explicitly,
\begin{equation} \label{eq:LambdaMij}
    \begin{split}
        \lambda(M_{ij})
        & = \frac{1}{2}(\mu w_i + \mu w_j - \lambda_i - \lambda_j) \\
        & \pm \sqrt{\frac{1}{4}(\mu w_i + \mu w_j - \lambda_i
        - \lambda_j)^2 - \mu(w_i - w_j)(\lambda_j-\lambda_i)},
    \end{split}
\end{equation}
and both of them are positive. Thus, \eqref{eq:Z1Part} has the lower and
upper bounds, 
\begin{subequations}\label{eq:BoundsOfZ1Part}
    \begin{align}
        \label{eq:BoundsOfZ1Part:a}
        \eqref{eq:Z1Part} & \geqslant \tr(Z_1^\top Z_1)
        \cdot\min\left\{2(\mu w_p - \lambda_p), 
        \min\limits_{i<j}\lambda_{\min}(M_{ij})\right\},\\
        \label{eq:BoundsOfZ1Part:b}
        \eqref{eq:Z1Part} & \leqslant \tr(Z_1^\top Z_1)
        \cdot\max\left\{2(\mu w_1 - \lambda_1), 
        \max\limits_{i<j}\lambda_{\max}(M_{ij})\right\}.
    \end{align}
\end{subequations}
Again, both bounds in~\eqref{eq:BoundsOfZ1Part} are achievable.
In~\eqref{eq:BoundsOfZ1Part:a}, the inequality is saturated if $Z_1$ is
parallel to $(0,\dotsc,0,e_p)$, or if $(z_{ij},z_{ji})^\top$ is the
eigenvector corresponding to $\min_{i<j}\lambda_{\min}(M_{ij})$ and other
entries of $Z_1$ are zero. Similarly in~\eqref{eq:BoundsOfZ1Part:b}, the
equality is satisfied if $Z_1$ is parallel to $(e_1,0,\dotsc,0)$, or if
$(z_{ij},z_{ji})^\top$ is the eigenvector corresponding to
$\max_{i<j}\lambda_{\max}(M_{ij})$ and other entries of $Z_1$ are zero.
Putting all bounds in~\eqref{eq:BoundsOfZ2Part}
and~\eqref{eq:BoundsOfZ1Part} together, we proved
\eqref{eq:ConditionNumber}.
\end{proof}

Theorem~\ref{thm:ConditionNumber} gives the exact condition number at
global minimizers, which provides an estimation of the local convergence
rate around the optima. Based on \eqref{eq:ConditionNumber}, we could
estimate the local convergence if the weight matrix $W$ is chosen. In
order to minimize the condition number of Hessian operator, we will
provide an intuitive approach to select a near-optimal weight matrix in
the next section.

\begin{remark}

We give a discussion for matrices with degenerate eigenvalues among
$\lambda_1, \lambda_2, \dots, \lambda_p$,
i.e., we relax the assumption~\eqref{eq:SpectrumOfA} as
\begin{equation} \label{eq:RelaxedSpectrumOfA}
    \lambda_1 \leqslant \lambda_2 \leqslant \cdots \leqslant \lambda_p
    < \lambda_{p+1} \leqslant \cdots \leqslant \lambda_n.
\end{equation}
Theorem~\ref{thm:StationaryPoint} and Theorem~\ref{thm:NoLocalMinima}
remains valid in their current forms. Theorem~\ref{thm:GlobalMin} is also
valid up to some changes due to the non-uniqueness of the eigenvectors of
$A$. We give an example to show the idea of the required changes. Assume
$\lambda_i = \lambda_{i+1}$ for $i < p$. Then any vector $v \in
\mathrm{span}\{v_i, v_{i+1}\}$ is an eigenvector of $A$ corresponding to
$\lambda_i$ and $\lambda_{i+1}$. Therefore, the matrix $V_p$ in
Theorem~\ref{thm:GlobalMin} may not be in the current form, $V = (v_1,
v_2, \dotsc, v_n)$. Instead, the matrix $V_p$ could be changed to
\begin{equation*}
    V_p = (v_1, v_2, \dotsc, v_p) \cdot 
    \begin{pmatrix}
        I_{i-1} & & \\
        & Q_1 & \\
        & & I_{n-i-1}
    \end{pmatrix},
\end{equation*}
where $Q_1$ is a $2 \times 2$ orthogonal matrix. More generally, if there
are $r$ distinct eigenvalues among $\{\lambda_1, \lambda_2, \dotsc,
\lambda_p \}$, the matrix $V_p$ in Theorem~\ref{thm:GlobalMin} admits the
form
\begin{equation*}
    V_p = (v_1, v_2, \dotsc, v_p) \cdot 
    \begin{pmatrix}
        Q_1 & & & \\
        & Q_2 & & \\
        & & \ddots & \\
        & & & Q_r
    \end{pmatrix},
\end{equation*}
where $Q_i$ is an orthogonal matrix with dimension being the degree of
degeneracy of the $i$th distinct eigenvalues.

However, as for Theorem~\ref{thm:ConditionNumber}, since the global
minimizers are not isolated, the Hessian matrix of $\fmu$ at $X^*$ cannot
be positive definite. We could employ techniques in
Section~\ref{sec:hermitianmatrices} to remove the nonzero null space in
the Hessian matrix.
\end{remark}

\subsection{Near-Optimal Weight Matrix}
\label{sec:optimalW}

According to the analysis above, the parameter $\mu$ and the weight matrix
$W$ could be considered as an ensemble $\mu W$. That means that the degree
of freedom of the parameters in~\eqref{eq:WeightedTracePenalty} is $p$
instead of $p+1$. Therefore, we can always set that $\mu = 1$ in analysis.
Given the exact expression of the condition number $\kappa\left(\hess \fmu
(X^*)\right)$, we would like to minimize the condition number with respect
to the weight matrix $W$ and obtain the optimal weight matrix $W^*$, \ie, 
\begin{equation} \label{eq:MinW}
    W^* = \argmin_W 
    \frac{
        \max\left\{\lambda_n-\lambda_1,2( w_1 - \lambda_1), 
        \max\limits_{i<j}\lambda_{\max}(M_{ij})\right\}
        }
        {
        \min\left\{\lambda_{p+1}-\lambda_p,2(w_p - \lambda_p), 
        \min\limits_{i<j}\lambda_{\min}(M_{ij})\right\}
        }.
\end{equation}
However, \eqref{eq:MinW} relies on the eigenvalues of $A$, which is not
known \textit{a priori}. Hence
solving \eqref{eq:MinW} exactly is infeasible.

In the following, we derive a near-optimal solution to \eqref{eq:MinW}
with eigenvalues of $A$. The final near-optimal solution relies on the
relative eigenvalue distribution of $A$, which is known
\textit{a priori} in many
practical applications, \eg, FCI. In quantum chemistry, an FCI
solver is usually applied after Hartree-Fock calculation, which offers a
good estimation of the
eigenvalues~\cite{
Slater_1951_ASimplificationHF}. For general eigenvalue problems, we could
estimate the eigenvalues of $A$ at a lower cost than
eigensolvers~\cite{Lin_2016_ApproximatingSpectralDensities}.

To simplify the later discussion, we assume that $\lambda_p +
\lambda_{p+1} < \lambda_1 + \lambda_n$, which in almost all practical
applications is satisfied if $n \gg p$. We choose the weight matrix $W$
such that
\begin{equation} \label{eq:ChooseWCondition}
    \frac{\lambda_1 + \lambda_n}{2} \geqslant w_1 > \ldots > w_p
    \geqslant \frac{\lambda_p + \lambda_{p+1}}{2}.
\end{equation}
Then the objective function in \eqref{eq:MinW} can be simplified as
\begin{equation} \label{eq:ApproximateConditionNumber}
    \kappa\left(\hess \fmu (X^*)\right) = 
    \frac{
        \max\left\{\lambda_n-\lambda_1, 
        \max\limits_{i<j}\lambda_{\max}(M_{ij})\right\}
        }
        {
        \min\left\{\lambda_{p+1}-\lambda_p, 
        \min\limits_{i<j}\lambda_{\min}(M_{ij})\right\}
        }.
\end{equation}
Recall that the eigenvalues of $M_{ij}$ as in \eqref{eq:LambdaMij} admit
\begin{equation*}
    \begin{split}
        \lambda(M_{ij})
        = & \frac{1}{2}(w_i+w_j - \lambda_i -\lambda_j)
        \left(1 \pm \sqrt{
            1-\frac{
                4(w_i - w_j)(\lambda_j-\lambda_i)
                }
                {
                (w_i+w_j - \lambda_i -\lambda_j)^2
                }
                }
        \right) \\
        \leqslant & w_i + w_j - \lambda_i - \lambda_j
        \leqslant \lambda_n - \lambda_1
    \end{split}
\end{equation*}
where the second inequality follows from~\eqref{eq:ChooseWCondition}. Thus,
to find the minimizer of \eqref{eq:MinW} means to maximize the
denominator, whose difficulty lies in solving
\begin{equation} \label{eq:MaxMinEigenvalueMij}
    \max_{w_1>\cdots > w_p} \min_{1\leqslant i<j\leqslant p}
    \frac{1}{2}(w_i+w_j - \lambda_i -\lambda_j)
    \left(1- \sqrt{
        1-\frac{
            4(w_i - w_j)(\lambda_j-\lambda_i)
            }
            {
            (w_i+w_j - \lambda_i -\lambda_j)^2
            }
            }
    \right).
\end{equation}
Exactly solving \eqref{eq:MaxMinEigenvalueMij} remains complicated. Here
we give an intuitive analysis. Notice that, due to
\eqref{eq:ChooseWCondition}, $(w_i + w_j - \lambda_i - \lambda_j)$ is
lower bounded by $\lambda_{p+1} - \lambda_p$. When $\frac{4(w_i - w_j)
(\lambda_j-\lambda_i)}{(w_i+w_j - \lambda_i -\lambda_j)^2} > c$ for all
$i<j$ and $c > 0$ is a constant bounded away from zero, then we have
$\kappa\left(\hess \fmu (X^*)\right) < \frac{\lambda_n - \lambda_1}{\lambda_{p+1} -
\lambda_p}\frac{2}{1 - \sqrt{1-c}}$. When some $\frac{4(w_i - w_j)
(\lambda_j-\lambda_i)}{(w_i+w_j - \lambda_i -\lambda_j)^2}$ approaches
zero, \eg, the eigengap $\lambda_j - \lambda_i$ is small, we apply the
linear approximation to the square root term in
\eqref{eq:MaxMinEigenvalueMij} and obtain,
\begin{equation} \label{eq:LinearApproxMaxMin}
    \max_{w_1>\cdots > w_p} \min_{1\leqslant i<j\leqslant p}
    \frac{(w_i - w_j)(\lambda_j-\lambda_i)}
    {(w_i+w_j - \lambda_i -\lambda_j)}.
\end{equation}
Solving the max-min problem~\eqref{eq:MaxMinEigenvalueMij}
exactly is difficult. Since \eqref{eq:ChooseWCondition} give lower and upper
bounds of the denominator in~\eqref{eq:LinearApproxMaxMin},  
we only focus on optimizing the numerator part, 
\begin{equation*}
    F(W) = \max_{w_1>\cdots > w_p} \min_{1\leqslant i<j\leqslant p}
    (w_i - w_j)(\lambda_j-\lambda_i),
\end{equation*}
which has the analytical solution $\widehat{W}$ satisfying
\begin{align*}
    \hat{w}_1 &= \frac{\lambda_1 + \lambda_n}{2},\\
    \hat{w}_p &= \frac{\lambda_p + \lambda_{p+1}}{2},\\
    \hat{w}_i - \hat{w}_{i+1} &= \left(
        \sum_{j=1}^{p-1} (\lambda_{j+1}-\lambda_j)^{-1}
        \right)^{-1} \frac{\hat{w}_1-\hat{w}_p}{\lambda_{i+1}-\lambda_i},\\
    F(\widehat{W}) &= \left(
        \sum_{j=1}^{p-1}(\lambda_{j+1}-\lambda_j)^{-1}
        \right)^{-1}(\hat{w}_1-\hat{w}_p).
\end{align*}
Furthermore, through a simple calculation, we obtain that
\begin{equation*}
    \frac{F(\widehat{W})}{p-1} \leqslant F(\widetilde{W})
    \leqslant F(\widehat{W}),
\end{equation*}
where the weight matrix $\widetilde{W}$ is evenly distributed between
$\hat{w}_p$ and $\hat{w}_1$. Such an inequality indicates that the uniform
weight matrix $\widetilde{W}$ is a simple but effective choice for small
$p$ because it could use a few eigenvalues known \textit{a
priori} to determine a weight matrix with a controlled condition number.
Later in Section~\ref{sec:numericalresult}, all the numerical experiments
use the evenly distributed weight matrix $\widetilde{W}$ since in practice
it needs no extra cost. 

\subsection{Generalization for Hermitian Matrices}
\label{sec:hermitianmatrices}

In this section, we will discuss the extension of
\eqref{eq:WeightedTracePenalty} for the eigenvalue problem of complex Hermitian
matrices. Given that $A$ is an Hermitian matrix and $(\Lambda, V)$ is the
eigenpairs such that
\begin{equation} \label{eq:ComplexConditionOfA}
    A = V \Lambda V^\conj,
\end{equation}
where $V = (v_1, v_2, \dots, v_n) \in \bbC^{n\times n}$ and $\Lambda =
\diag(\lambda_1, \lambda_2, \dots, \lambda_n)$ such that
\begin{equation*}
    \lambda_1 < \lambda_2 < \cdots < \lambda_p < 
    \lambda_{p+1} \leqslant \cdots \leqslant \lambda_n.
\end{equation*}
We generalize the weighted trace penalty
model~\eqref{eq:WeightedTracePenalty} for complex Hermitian matrices as
\begin{equation} \label{eq:ComplexModel}
    \min_{X\in \bbC^{n \times p}}
    \fmu(X) = \frac{1}{2}\tr(X^\conj AX)
    + \frac{\mu}{4}\fnorm{X^\conj X - W}^2,
\end{equation}
where the conditions on $\mu$ and $W$ remain unchanged, \ie, $\mu$ is a
positive scalar and $W$ is a real diagonal matrix that satisfies
Assumption~\ref{assump:ConditionOfW}. The first-order optimal condition is
\begin{equation} \label{eq:ComplexFirstOrderCondition}
    \nabla \fmu (X) = AX + \mu X(X^{\mathrm{H}}X-W) = 0.
\end{equation}
Theorems~\ref{thm:StationaryPoint} and~\ref{thm:GlobalMin} could be
generalized to complex matrices, which are detailed in
Theorem~\ref{thm:ComplexStationaryPoint} and~\ref{thm:ComplexGlobalMin}.
The proofs of these theorems remain similar to the cases of real matrices. 

\begin{theorem}
    \label{thm:ComplexStationaryPoint}
    Assume $A$ and $W$ satisfy \eqref{eq:ComplexConditionOfA} and
    Assumption~\ref{assump:ConditionOfW} respectively. Any stationary
    point $\widehat{X}$ of \eqref{eq:ComplexModel} has the form
    \begin{equation}
        \widehat{X} = \widehat{U}_p \widehat{S}_p,
    \end{equation}
    where $\widehat{U}_p = (\hat{u}_1, \hat{u}_2, \dots, \hat{u}_p) \in
    \bbC^{n\times p}$ and $\widehat{S}_p = \diag(\hat{s}_1, \hat{s}_2,
    \dots, \hat{s}_p) \in \bbR^{p\times p}$ such that 
    \begin{equation}\label{eq:ComplexPropertyStationaryPoint}
        \begin{split}
            A \hat{u}_i = \sigma_i \hat{u}_i,
            \quad \hat{u}_i^{\mathrm{H}} \hat{u}_j = \delta_{ij} \text{ and,} \\
            \hat{s}_i \in \left\{0, \sqrt{w_i - \frac{\sigma_i}{\mu}}
            \right\}.
        \end{split}
    \end{equation}
\end{theorem}

\begin{theorem}
    \label{thm:ComplexGlobalMin}
    Assume $A$ and $W$ satisfy \eqref{eq:ComplexConditionOfA} and
    Assumption~\ref{assump:ConditionOfW} respectively. The global
    minimizer $X^*$ of \eqref{eq:ComplexModel} has the form,
    \begin{equation*}
        X^* = V_p S_p,
    \end{equation*}
    where $V_p = (v_1 e^{\imath \theta_1}, v_2 e^{\imath
        \theta_2}, \dotsc, v_p e^{\imath \theta_p})$ for every $\theta_i
        \in \bbR$, and $S_p = \diag(s_1, s_2, \dots, s_p) \in
        \bbR^{p \times p}$ such that 
    \begin{equation*}
      s_i^2 = w_i - \frac{\lambda_i}{\mu}. 
    \end{equation*}
\end{theorem}

However, the properties of the Hessian operator for \eqref{eq:ComplexModel},
\begin{equation} \label{eq:ComplexHessian}
    \hess \fmu(X)[C] = A C + \mu (C X^\conj X + X C^\conj X
    + X X^\conj C - C W),
\end{equation}
change dramatically. The major difference is that the Hessian operator
\eqref{eq:ComplexHessian} is no longer a positive definite operator,
instead, it is positive semi-definite. From
Theorem~\ref{thm:ComplexGlobalMin}, we find that any global minimizer
$X^*$ multiplied by a phase rotation $e^{\imath \theta}$ remains a global
minimizer. Hence, unlike the symmetric matrix case where global minimizers
are isolated, for Hermitian matrices, the global minimizers are located on
a circle of the $np$-dimensional complex space and the Hessian operator at
these global minimizers is positive semi-definite but not positive
definite. In the following, we update our previous results and extend them
for Hermitian matrices.

Define the inner product of $X \in \bbC^{n \times p}$ and $Y \in \bbC^{n
\times p}$ as
\begin{equation*}
    \left\langle X, Y \right\rangle \triangleq \Re [\tr (X^\conj Y)],
\end{equation*}
where $\Re[\cdot]$ denotes the real part of a complex number.
For any global minimizer $X^*$,
each element in
\begin{equation}\label{eq:Manifold}
    \left\{X \in \bbC^{n \times p} \mid X = X^* e^{\imath \Theta},
    \Theta = \diag(\theta_1,\theta_2,\dots,\theta_p),
    \theta_i \in \bbR \right\},
\end{equation}
is still a global minimizer.
The tangent vectors of the manifold~\eqref{eq:Manifold} is 
\begin{equation*}
    T \triangleq \left\{\imath X^* \Gamma \in \bbC^{n\times p} \mid
    \Gamma = \diag(\gamma_1,\gamma_2,\dotsc,\gamma_p) \in \bbC^{p \times p}
    \right\}.
\end{equation*}
Then, we constrain the condition number of the Hessian operator on the
perpendicular manifold
\begin{equation} \label{eq:ManifoldConstrain}
    T^\bot = \left\{C \in \bbC^{n \times p} \mid
    \Im [C^\conj X^*]_{ii} = 0, \forall i \right\},
\end{equation}
where the subscript $ii$ denoted the $i$-th diagonal element and
$\Im[\cdot]$ denotes the imaginary part. Finally, we would like to
estimate the lower and upper bounds of the inner product, 
\begin{multline} \label{eq:ComplexConstrainedInnerProd}
    \left\langle C, \hess \fmu (X^*) [C] \right\rangle \Big|_{C \in T^\bot}
    = \tr(C^\conj A C) - \tr(C^\conj C \Lambda_p)\\
    + \tr(\mu C^\conj X^*(X^*)^\conj C)
    + \Re \left[\tr(\mu C^\conj X^* C^\conj X^*) \right]. 
\end{multline} 
Similar to the proof of Theorem~\ref{thm:ConditionNumber},
\eqref{eq:ComplexConstrainedInnerProd} is upper and lower bounded by the
numerator and the denominator in \eqref{eq:ConditionNumber:b} respectively.

Now, we are going to explain the reason behind splitting the space into
$T$ and $T^\bot$. Considering the gradient $\nabla \fmu(X)$
in~\eqref{eq:ComplexFirstOrderCondition}, it can be represented near the
global minimizer as
\begin{equation}\label{eq:ComplexGradientDecompose}
    \nabla \fmu (X) = G_X + \varDelta G_X
\end{equation}
where $G_X \in \{C \in \bbC^{n \times p}: \Im [C^{\mathrm{H}}X^*]_{ii} =
0,\;\forall i\}$ and $\lVert\varDelta G_X\rVert$ is $o(\norm{G_X})$. Let
$X = X^* + \varDelta X$. Consequently, without loss of generality let $\mu
= 1$ and \eqref{eq:ComplexGradientDecompose} can be shown by 
\begin{align}
    \nabla \fmu(X) 
    &= A(X^* + \varDelta X) + 
    (X^* + \varDelta X)\big((X^* + \varDelta X)^{\mathrm{H}}(X^*
    + \varDelta X) - W\big),\\
    &= A \cdot \varDelta X + X^* \big((X^*)^{\mathrm{H}} \varDelta X
    + \varDelta X^{\mathrm{H}} X^*\big)
    +\varDelta X\big((X^*)^{\mathrm{H}} X^* - W\big)\nonumber \\
    & \quad + o(\norm{\varDelta X}), \nonumber\\
    & \triangleq G_X + \varDelta G_X,
\end{align}
where the second equality adopts~\eqref{eq:ComplexFirstOrderCondition}.
Thus, 
\begin{align}\label{eq:XG}
    (X^*)^\conj G_X &= (X^*)^\conj A \varDelta X + (X^*)^\conj X^*
    \big((X^*)^{\mathrm{H}} \varDelta X + \varDelta X^{\mathrm{H}}
    X^*\big) \\
    & \quad +(X^*)^\conj \varDelta X\big((X^*)^{\mathrm{H}} X^*
    - W\big), \nonumber \\
    &= (W - \lvert S_p \rvert^2) (X^*)^\conj \varDelta X + (X^*)^\conj
    \varDelta X (\lvert S_p \rvert^2 - W) \nonumber \\
    &\quad + \lvert S_p \rvert^2 \big((X^*)^{\mathrm{H}} \varDelta X
    + \varDelta X^{\mathrm{H}} X^*\big), \nonumber
\end{align}
where $\lvert S_p \rvert^2 = \diag \big(\abs{s_1}^2, \dotsc ,
\abs{s_p}^2\big)$. It is revealed that the diagonal elements of the third
term are real, and the diagonal parts of the first two terms are opposite,
which means the diagonal of~\eqref{eq:XG} is real and $G_X \in T^\bot$.

That indicates the gradient in the neighborhood of the global minimizer is
located on the manifold~\eqref{eq:ManifoldConstrain} dominantly, and when
we use a gradient descent method to solve the optimization, the
convergence rate is mainly dependent on the constrained condition number.
Though the Hermitian matrix is interesting in some applications, in our
target applications, the matrices are real symmetric. Hence, we omit the
detail for the analysis of Hermitian matrices.

\section{Algorithms}
\label{sec:algorithm}

In this section, we introduce several algorithms to address the
unconstrained nonconvex minimization
problem~\eqref{eq:WeightedTracePenalty} for large-scaled matrices $A$.

\subsection{Gradient Descent Methods}

A common choice for our unconstrained minimization problem
\eqref{eq:WeightedTracePenalty} is the gradient descent method. A general
gradient descent method with various stepsize strategies admits the form,
\begin{equation} \label{eq:GDForm}
    X^{(j+1)} = X^{(j)} - \alpha^{(j)} \nabla \fmu (X^{(j)}),
\end{equation}
where the superscript $(j)$ denotes the iteration index and $\alpha^{(j)}$
is the stepsize at the $j$-th iteration. Different stepsize strategies
lead to different convergence properties. We first consider a fixed
stepsize that is sufficiently small. As shown in
Theorem~\ref{thm:NoLocalMinima}, the weighted trace penalty model
\eqref{eq:WeightedTracePenalty} does not have spurious local minima. We
could then first adopt the idea in
\cite{gaoTriangularizedOrthogonalizationfreeMethod2020} to guarantee that
the iteration never escapes from a big area such that the Lipshitz
constant is bounded. Then by the discrete stable manifold
theorem~\cite{gaoGlobalConvergenceTriangularized2021, Lee2019}, we could
show that the gradient descent method for \eqref{eq:WeightedTracePenalty}
converge to global minima for all initial points besides a set of measure
zero. Consider another choice of stepsize strategy, \ie, random
perturbation of a fixed stepsize. Instead of using discrete stable
manifold theorem, we could apply ideas from stable manifold theorem for
random dynamical systems~\cite{Chen2021b} to show global convergence
almost surely. Wen et al.~\cite{wen2016trace} showed that, if the stepsize
is sufficiently small, not necessarily constant, the iteration variable of
the gradient descent method for the original trace penalty model stays
full-rank. Such a result could be extended to our weighted trace penalty
model as well. Though aforementioned stepsize strategies, in theory, work
well on global convergence. In practice, these stepsize strategies are too
conservative to be numerically efficient.

For most applications, especially the FCI eigenvalue problem we considered
in this paper, a good initialization is available, and, hence, more
aggressive stepsize strategies are adopted in practice. Such stepsize
strategies include but are not limited to exact line search, BB stepsize,
etc. For the weighted trace penalty model, the stepsize $\alpha^{(j)}$ can
be computed by exact line search to make $X^{(j+1)}$ attain local
directional optima in each iteration, which leads to a cubic polynomial of
$\alpha$ as the sub-problem. Numerically, we find that BB stepsize works
better in the gradient descent method for our weighted trace penalty
model. Therefore, we mainly focus on the BB stepsize. Let $\delta_X^{(j)}
\triangleq X^{(j)} - X^{(j-1)}$ and $\delta_G^{(j)} \triangleq \nabla \fmu
(X^{(j)})-\nabla \fmu (X^{(j-1)}).$ The BB stepsize is defined as,
\begin{align*}
    \alpha_{\text{odd}}^{(j)} = \tr\left((\delta_X^{(j)})^\top
    \delta_G^{(j)}\right)\Big/\fnorm{\delta_G^{(j)}}^2,\\
    \alpha_{\text{even}}^{(j)} = \fnorm{\delta_X^{(j)}}^2\Big/
    \tr\left((\delta_X^{(j)})^\top \delta_G^{(j)}\right).
\end{align*}
where the subscripts odd and even mean the iteration number $(j)$ is odd
or even. The BB stepsize requires some extra storage to store intermediate
matrices $\delta_X^{(j)}$ and $\delta_G^{(j)}$ (or their variants), and
the computational cost for the stepsize is $O(np)$. As a comparison, in
the gradient descent method, the dominant per-iteration computational cost
is to compute the matrix-matrix product of $AX$, which costs about
$O(\mathrm{nnz}(A)\cdot p)$ for $\mathrm{nnz}(A)$ denoting the number of
nonzero entries in $A$.

Compared with the original trace-penalty optimization~\cite{wen2016trace},
the only change is the weight matrix. Introducing such a weight matrix
reduces the cardinality of the global minima set from infinite to finite
and makes all global minima isolated from each other, while searching for
the global minima becomes more difficult. This could be seen from the
theoretical condition numbers of the Hessian operator
in~\eqref{eq:ConditionNumber} and that in~\cite{wen2016trace}:
\begin{equation}
    \kappa\left(\hess \fmu\right) \geqslant
    \kappa \left(\hess f_{\mu,I} \Bigr|_{V_p^{\bot}}\right)
    \triangleq \frac{\lambda_n - \lambda_1}{\lambda_{p+1} - \lambda_p}.
\end{equation}
Viewing both optimization methods as eigensolvers, the original
trace-penalty optimization requires an extra Rayleigh-Ritz process,
whereas the weighted trace penalty model converges to desired eigenpairs
directly.

\subsection{Coordinate Descent for FCI}\label{subsec:CD}

The gradient descent method for \eqref{eq:WeightedTracePenalty} works well
for problems of small sizes to moderate sizes. While, for FCI matrices,
the gradient descent method becomes less efficient and, in many cases,
infeasible. In this section, we will introduce the coordinate descent
method to optimize~\eqref{eq:WeightedTracePenalty}.

Recall that an FCI matrix has the following properties:
\begin{itemize}
    \item Extremely large-scale: in practice, the dimension of the FCI
    matrix could easily exceed $O(10^{14})$. This makes the eigenvectors
    impossible to be stored in memory. The FCI matrix itself has to be
    generated on the fly, and we cannot keep the whole matrix in memory
    but calculate one column or row when we use it.

    \item Sparsity: the Hamiltonian operator under the
    second-quantization~\cite{Berezin_1966_MethodSecondQuantization}
    admits 
    \begin{equation*}
        \hat{H} = \sum_{p,q} t_{pq} \hat{a}_p^{\dagger} \hat{a}_q
        + \frac{1}{2}\sum_{p,q,r,s} u_{pqrs} \hat{a}_p^{\dagger}
        \hat{a}_q^{\dagger} \hat{a}_s \hat{a}_r,
    \end{equation*}
    where $\hat{a}_p^{\dagger}$ and $\hat{a}_p$ denote the creation and
    annihilation operators of an electron with spin-orbital index $p$, and
    $t_{pq}$ and $u_{pqrs}$ are one- and two-electron integrals, the
    $(i,j)$-th element of FCI matrices is nonzero if and only if the
    $i$-th basis wavefunction and the $j$-th basis wavefunction differ in
    at most two occupied
    spin-orbitals~\cite{Condon_1930_TheoryComplexSpectra,
    Slater_1929_TheoryComplexSpectra}. Thus, the number of nonzero
    elements grows polynomially with respect to the number of particles
    whereas the FCI matrix size grows factorially.

    \item Approximately sparse eigenvectors: the eigenvectors associated
    with low-lying eigenvalues of FCI matrices usually are sparse. The
    magnitudes of different entries vary widely, ranging from $10^{-16}$
    to $10^{-1}$ in normalized eigenvectors. Only a few dominant entries
    account for nearly all the norms of eigenvectors. In practice, we
    approximate these eigenvectors by sparse vectors and focus on those
    dominant entries.
\end{itemize}
Taking all these properties into account, for symmetric FCI matrix $A$,
the gradient descent method is not feasible: the matrix $A$ and iteration
variable $X$ cannot be hosted in memory; and, computing $AX$ each
iteration is not affordable. CDFCI~\cite{wangCoordinateDescentFull2019a},
solving the leading eigenpair of the FCI problem, inspires us that the
coordinate descent method would be an efficient algorithm to locate the
sparse entries and calculate the values.

When a coordinate descent method is considered, the per-iteration updating
scheme for \eqref{eq:WeightedTracePenalty} is of the form 
\begin{subequations} \label{eq:CDForm}
\begin{align}
    & \text{Pick a coordinate from $X^{(j)}$, \ie, } (k^{(j)},\ell),
    \label{eq:CDFormPick}\\
    & \alpha^{(j)} = \argmin_{\alpha \in \bbR}
    \fmu \left(X^{(j)} + \alpha E_{k^{(j)}\ell}\right),
    \label{eq:CDFormStepsize}\\
    & X^{(j+1)} = X^{(j)} + \alpha^{(j)} E_{k^{(j)}\ell},
\end{align}
\end{subequations}
where $E_{k^{(j)}\ell} \in \bbR^{n \times p}$ denotes the matrix whose
$(k^{(j)},\ell)$-th element is one and zero elsewhere. The optimization
problem in \eqref{eq:CDFormStepsize} is a fourth order polynomial of
$\alpha$, whose minimizers can be obtained via solving a cubic polynomial
directly. The cubic polynomial is of form 
\begin{equation*}
    (\alpha + x_{k \ell})^3 + c_1 (\alpha + x_{k\ell}) + c_0 = 0.
\end{equation*}
where the coefficients are 
\begin{equation}\label{eq:CoefCubicEq}
    \begin{split}
        c_1 & = \frac{1}{\mu} a_{kk} - w_\ell
        + \sum_{m=1}^n (x_{m\ell})^2 
        + \sum_{m=1}^p (x_{km})^2 - 2(x_{k\ell})^2,\\
        c_0 &= \frac{1}{\mu} \sum_{m=1}^n
        a_{km}x_{m\ell} - \frac{a_{kk}x_{k\ell}}{\mu}
        + \sum_{m=1}^n \sum_{s=1}^p x_{ks}x_{ms}x_{m\ell}\\
        &\quad + x_{k\ell}^3 - x_{k\ell}\Bigl(\sum_{s=1}^p (x_{ks})^2
        + \sum_{m=1}^n (x_{m\ell})^2\Bigr).
    \end{split}
\end{equation}
Since all variables in the above two equations are at $j$-th iteration, we
drop the iteration index superscript for all variables. As we shall see in
the later Algorithm~\ref{alg:CD}, we maintain $Y = AX$ and $S=X^\top X$
throughout iterations. Hence both coefficients can be computed in $O(p)$
operations and then the exact line search for stepsize $\alpha^{(j)}$ can
be calculated efficiently. 

Our coordinate picking strategy, \eqref{eq:CDFormPick}, is inspired by
CDFCI~\cite{wangCoordinateDescentFull2019a}, which depends on the gradient
and the nonzero pattern of $A$. In the $(j)$-th iteration, we focus on the
$\ell$-th column for $\ell \equiv j (\mathrm{mod}\, p)$. We search for the
entry with largest magnitude of the $\ell$-th column of $\nabla \fmu
(X^{(j)})$ among the nonzero pattern of $k^{(j-p)}$-th column of $A$,
where $k^{(j-p)}$ is the row coordinate updated in $(j-p)$-th iteration.
That is 
\begin{equation} \label{eq:PickingRule}
    k^{(j)} = \argmax_{i \in \calN \left(A_{:,k^{(j-p)}}\right)}
    \left\lvert \left( \nabla \fmu(X^{(j)}) \right)_{i \ell} \right\rvert,
\end{equation}
where $\calN(\cdot)$ denotes the nonzero pattern, $A_{:,k^{(j-p)}}$ denotes
the $k^{(j-p)}$-th column of $A$ and $\left(\nabla \fmu
(X^{(j)})\right)_{i\ell}$ denotes the $(i,\ell)$-th element of $\nabla
\fmu(X^{(j)})$.

According to the expression of $\nabla \fmu$ as in
\eqref{eq:FirstOrderCondition}, though only a small set of entries is
needed, computing them every iteration is not affordable. Thanks to an
important feature of the coordinate descent method, \ie, a single entry is
updated per-iteration, we could efficiently maintain two important
quantities: $Y^{(j)} \approx AX^{(j)}$ and $S^{(j)} =
\left(X^{(j)}\right)^\top X^{(j)}$. Similar to
CDFCI~\cite{wangCoordinateDescentFull2019a}, we maintain $Y^{(j)}$ as a
compressed approximation of $AX^{(j)}$. The updating combined with
compression formula is,
\begin{equation}\label{eq:UpdateAX}
    Y^{(j+1)}_{i \ell} =
    \begin{cases}
        Y^{(j)}_{i \ell} + \alpha^{(j)} A_{ik^{(j)}}
        & \text{if } \lvert \alpha^{(j)} A_{ik^{(j)}} \rvert > \varepsilon
        \text{ or } Y_{i\ell}^{(j)} \neq 0 \\
        Y^{(j)}_{i \ell} & \text{otherwise}
    \end{cases},
\end{equation}
for $i \in \calN \left(A_{:,k^{(j)}}\right)$. Besides, in order to get a stable
estimation of corresponding eigenvalues with high accuracy, we need to
recalculate the element $Y^{(j+1)}_{k^{(j)}\ell}$ exactly by
$Y^{(j+1)}_{k^{(j)}\ell} = (A_{:,k^{(j)}})^\top X_{:,\ell}$, since the
Rayleigh quotient has the updating scheme
\begin{equation}
    \label{eq:UpdateXAX}
    \left(X^\top A X\right)^{(j+1)}_{\ell\ell} = \left(X^\top A
    X\right)^{(j)}_{\ell\ell} + 2 \alpha^{(j)} Y^{(j+1)}_{k^{(j)}\ell}
    - (\alpha^{(j)})^2 A_{k^{(j)}k^{(j)}}.
\end{equation}
Due to the symmetry of $S^{(j)}$, only the upper-triangular part is stored
and updated, and the updating expression for $S^{(j)}$ is,
\begin{equation} \label{eq:UpdateXX}
    S_{im}^{(j+1)} = 
    \begin{cases}
        S_{im}^{(j)} + \alpha^{(j)} X_{k^{(j)}i}^{(j)}
        &\mathrm{if}\; i < \ell, m = \ell\\
        S_{im}^{(j)} + 2 \alpha^{(j)} X_{k^{(j)}i}^{(j)} + (\alpha^{(j)})^2
        &\mathrm{if}\; i = \ell, m = \ell\\
        S_{im}^{(j)} + \alpha^{(j)} X_{k^{(j)}m}^{(j)}
        &\mathrm{if}\; i = \ell, m > \ell\\
        S_{im}^{(j)}
        &\mathrm{otherwise}
    \end{cases}.
\end{equation}
The compression strategy of $Y^{(j)}$ restricts the increase of the number
of nonzero elements of both $Y$ and $X$. Compressing coordinates is very
much desired, which saves a significant amount of memory.
Algorithm~\ref{alg:CD} illustrates the framework of the algorithm.

\begin{algorithm}[htp]
    \caption{\WTPM{} by Coordinate Descent (\WTPM{}-CD)}\label{alg:CD}
    \begin{algorithmic}[1]
        \STATE Initialize $X^{(0)} \in \bbR^{n\times p}$, 
        penalty parameter $\mu$ and weight matrix $W$. 
        
        \STATE Store matrices $Y^{(0)} = AX^{(0)}$ and $S^{(0)} =
        (X^{(0)})^\top X^{(0)}$.
        
        \STATE Store $p$-dimensional vector
        $d^{(0)} = \textrm{the diagonal of } (X^{(0)})^\top Y^{(0)}$.
        
        \STATE Construct $\nabla \fmu(X^{(0)})$. Set $j=0$.

        \WHILE{stopping criterion not achieved}
            \STATE Pick the coordinate
            $(k^{(j)},\ell)$ to be updated in $(j+1)$-th iteration by
            \textbf{picking rule}~\eqref{eq:PickingRule}.

            \STATE Compute the coefficients $c_0, c_1$
            by~\eqref{eq:CoefCubicEq} and obtain the increment
            $\alpha^{(j)}$.

            \STATE $X^{(j+1)} = X^{(j)} +
            \alpha^{(j)} E_{{k^{(j)}}\ell}$.

            \STATE Update $Y^{(j+1)}$ by~\eqref{eq:UpdateAX}.

            \STATE Update $d^{(j+1)}$ by~\eqref{eq:UpdateXAX}.

            \STATE Update $S^{(j+1)}$ by~\eqref{eq:UpdateXX}.

            \STATE Construct the searching domain in $\nabla \fmu
            (X^{(j+1)})$ dependent on $\calN (A_{:,k^{(j)}})$.

            \STATE $j \leftarrow j+1$.

        \ENDWHILE
    \end{algorithmic}
\end{algorithm}

A good choice of initial point would make iterative methods efficient. For
FCI problems, Hartree-Fock provides excellent initial values for ground
states and a few low-lying excited states. We lack a systematic way of
choosing good initial vectors for other excited states. In principle, the
initial $X^{(0)}$ must be an extremely sparse matrix so that $Y^{(0)} =
AX^{(0)}$ could be calculated in a reasonable amount of time. In
particular, we adopt the following initialization for our numerical
results. We find the $p$ smallest elements and corresponding indices
$\{i_1,i_2,\dotsc,i_p\}$ in the diagonal of the FCI matrix. The initial
point $X^{(0)}$ is set to be $(e_{i_1}, e_{i_2},\dotsc, e_{i_p})$ where
$e_i$ denotes the $i$-th column of the identity matrix.

As for stopping criterion, in general, we use the residual norm $\fnorm{AX
- X\Lambda}$ where $\Lambda$ consists of the Ritz values or Rayleigh
quotients, or use the gradient norm $\fnorm{\nabla \fmu}$ to compare with
the tolerance $tol$. However, in these methods, high computational cost
arises from the extremely large size of $A$. The matrices $AX$ and $\nabla
\fmu$ gathered during the iteration are not exact due to inadequate update
of $Y^{(j)}$ and $\nabla \fmu (X^{(j)})$. Accordingly, we introduce the
summation of historical absolute increments as the stopping criterion at
$j$-th iteration, \ie,
\begin{equation}
    \label{eq:StoppingCriterion}
    \sum_{i=0}^h \gamma^i \abs{\alpha^{(j-i)}} < tol,
\end{equation}
where $h$ is a positive integer and $\gamma \in (0,1)$. Usually, we choose
$h = 100$ and $\gamma = 0.99$.

In $j$-th iteration, the cost of updating $X_{k^{(j)}\ell}$ dominantly
includes, selecting the index of largest elements $(k^{(j)},\ell)$ with
$O(\mathrm{nnz}(A_{:,k^{(j-p)}}))$ flops, updating $Y^{(j)}$ with
$O(\mathrm{nnz}(A_{:,k^{(j)}}))$ flops and inadequately constructing
$\nabla \fmu(X^{(j)})$ with $O(\mathrm{nnz}(A_{:,k^{(j)}}))$ flops. Thus
the coordinate descent method makes the computational cost affordable in
each iteration. Another advantage of the algorithm is the utilization of
memory. It only requires us to store some sparse matrices like $X^{(j)}$,
$Y^{(j)}$, a $p\times p$ matrix $S^{(j)}$ and incomplete $\nabla
\fmu(X^{(j)})$ in memory. 

In theory, the coordinate descent method converges faster than the full
gradient descent method. While, due to the fact that modern computer
architecture prefers batch operations, \ie, contiguous memory operations,
the full gradient descent method often outperforms the coordinate descent
method in runtime. However, the intrinsic structure of the FCI matrix
benefits most from the coordinate-wise method. The gradient descent method
has to access the sparse matrix and the related entries in $X$, which
destroys the contiguous memory access. On the other hand side, the
coordinate descent method allows us to compress coordinates and restrict
the cost of memory. Furthermore, the updating strategy provides more
chances for dominant elements to achieve their optimal values. Detailed
numerical results are provided in the next section.

\section{Numerical Experiments}
\label{sec:numericalresult}

In this section, we will test the performance of our algorithms for
computing a set of smallest eigenpairs of Hamiltonian matrices.

\subsection{Performance in Small Systems}

This section will discuss the performance of applying \WTPM{} to some
small systems and compare it with other eigensolvers. The FCI matrices are
illustrated in Table~\ref{tab:TestMatrix}. There are two matrices:
``ham448'' and ``h2o''. ``h2o'' matrix is generated from one \ch{H2O}
molecule system with STO-3G basis set and ``ham448'' matrix is generated
by the Hubbard model on a $4\times 4$ grid with $8$ fermions. The
dimension $n$ ranges from $6\times 10^4$ to $2\times 10^5$ which is much
smaller than the dimension of the systems of practical interest. That is
because we want to reveal the feature of \WTPM{} compared with the other
classical solvers. The average nnz$(A_{:,j})$ shows the number of nonzero
elements of each column in average, which roughly estimates the
computational expense of $O(\mathrm{nnz}(A_{:,j}))$ flops in each
iteration by \WTPM{}-CD. These numerical experiments on testing matrices
are performed in MATLAB R2021b.
\begin{table}[htp]
    \centering
    \caption{Testing FCI matrices.}\label{tab:TestMatrix}
    \begin{tabular}{ccccc}
    \hline
    \textbf{Name} & $n$ & nnz$(A)$ & average nnz$(A_{:,j})$& $\mathrm{nnz}(A)/n^2$\\ 
    \hline
    ham448 & 207168 & 32040806 & 155 &7.47e-4\\
    h2o & 61441 & 25060625 & 408 &6.64e-3\\
    \hline
    \end{tabular}
\end{table}

There still exists a problem how the $\mu$ and $W$ are determined in
practice. We present a feasible approach to choosing the proper parameters
based on roughly estimated eigenvalues. Since we usually could get a good
initial $X^{(0)}$ due to Hartree-Fock theory and this initial $X^{(0)}$
has only one nonzero element in each column, matrix $(X^{(0)})^\top  A
X^{(0)}$ with corresponding Rayleigh quotients $r_1 \leqslant r_2
\leqslant \dotsb \leqslant r_p$ can be computed cheaply. Just let $\mu =
1$ and $W$ is distributed evenly in the interval $[w_p,w_1]$ such that
\begin{equation}\label{eq:IntervalOfW}
    \begin{split}
        w_p & = r_p + \varepsilon,\\
        w_1 & = 2w_p - r_1,
    \end{split}
\end{equation}
where $\varepsilon$ is a positive parameter depending on both the
magnitude and the gap of initial Rayleigh quotients.

First, we demonstrate that the smallest eigenvectors are
approximately sparse in Figure~\ref{fig:SortEigVec}. In the accurate
eigenvectors, the magnitude of coordinates decreases quickly and the
vector norm is dominantly distributed on a few coordinates. It
implies that in the process of updating we could reserve the coordinates
whose magnitude is larger than a threshold and cut out others despite the
loss of accuracy. Figure~\ref{fig:SortEigVec} also shows the compression
threshold we used in the experiments.

\begin{figure}[htp]
    \centering 
    \subfloat[h2o]{\includegraphics[width=0.48\textwidth]{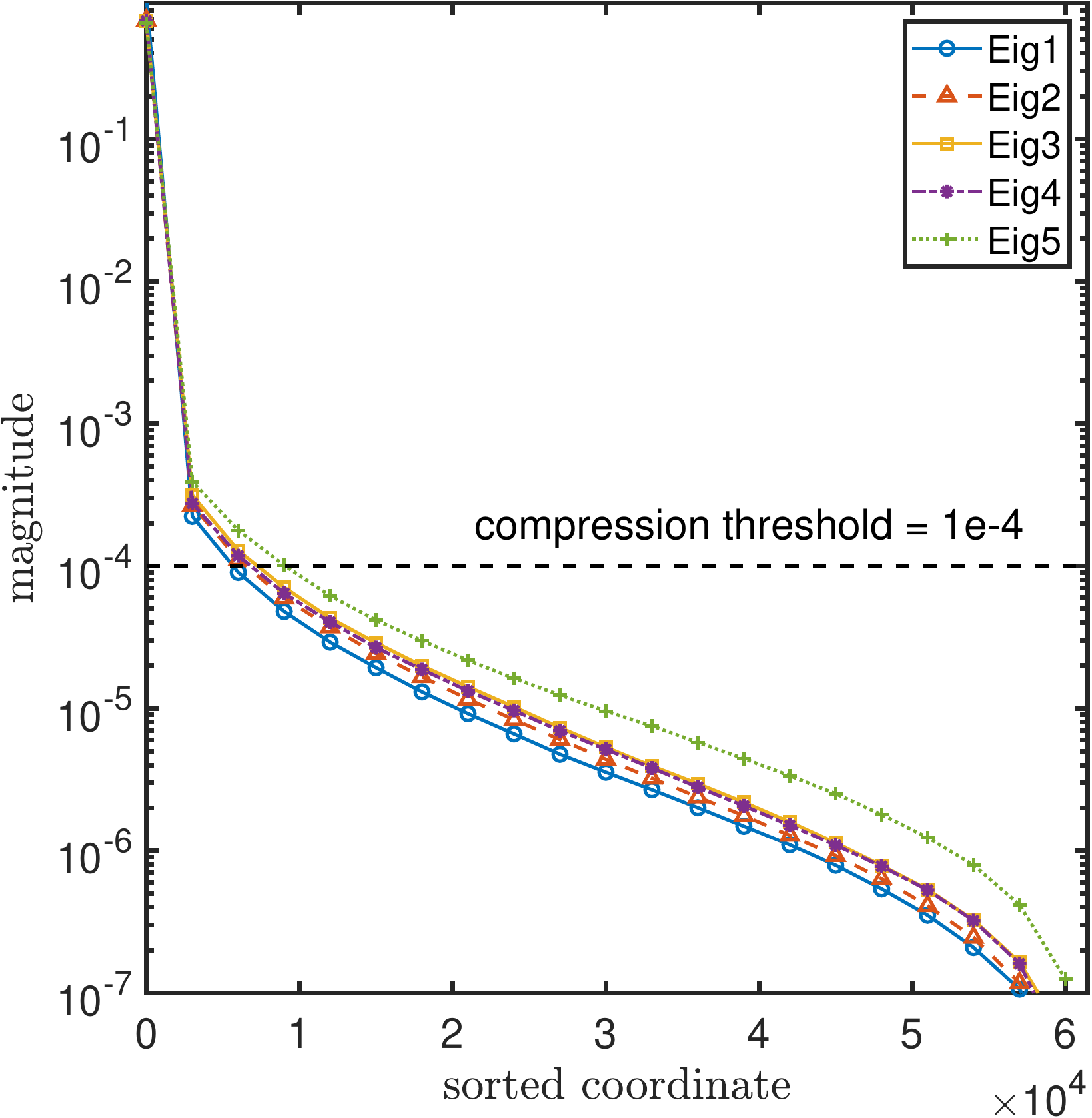}\label{fig:SortEigVec:a}}
    \quad
    \subfloat[ham448]{\includegraphics[width=0.48\textwidth]{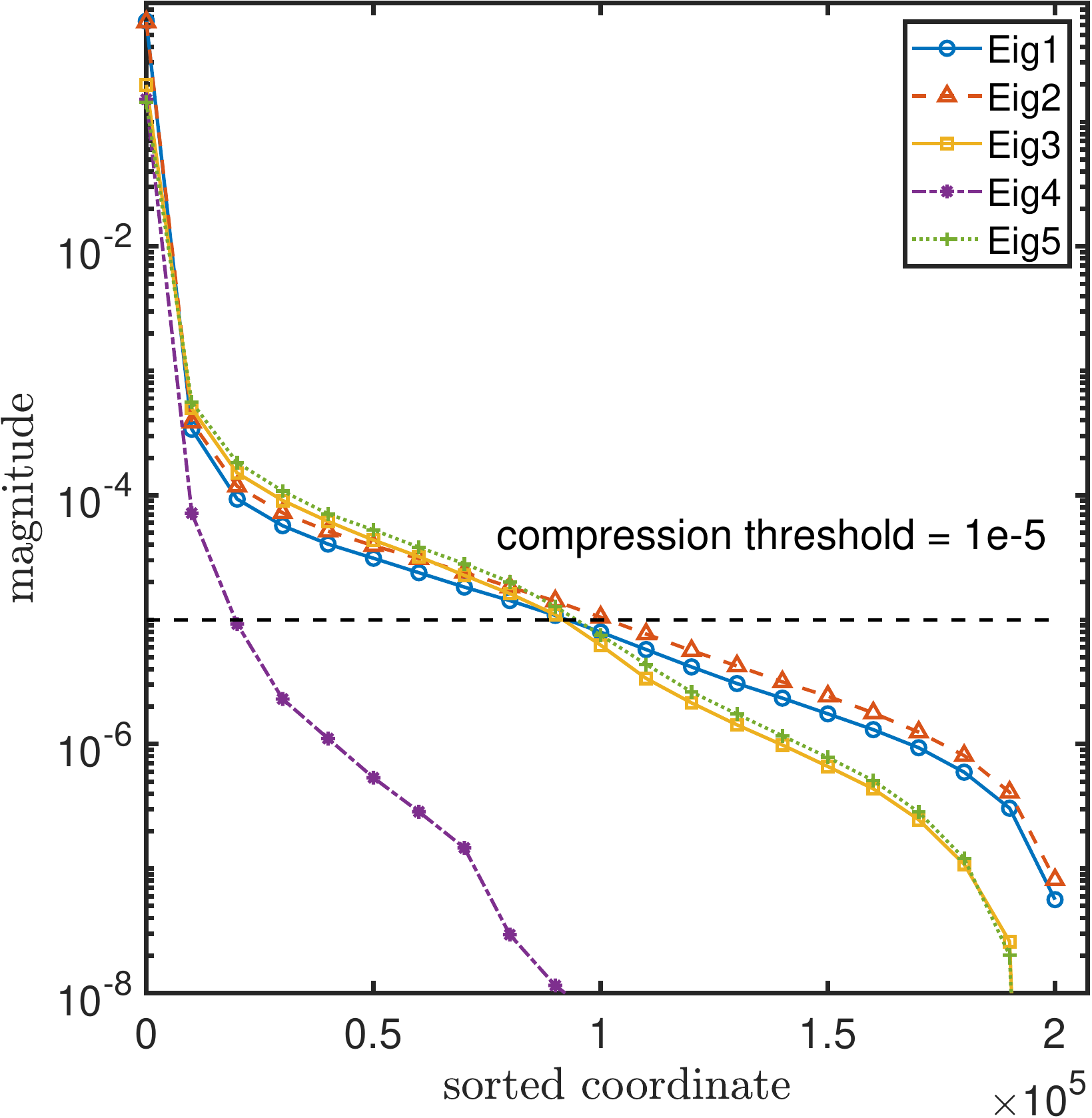}\label{fig:SortEigVec:b}}
    \caption{Absolute value of each coordinate in the sorted eigenvectors. 
    ``\textnormal{Eig}$k$'' denotes the $k$-th smallest eigenvector.}
    \label{fig:SortEigVec}
\end{figure}

We apply different eigensolvers to both FCI matrices, including LOBPCG in
BLOPEX toolbox~\cite{knyazev2001toward, Knyazev_2007_BlockLocallyOptimal},
EigPen-B introduced in trace-penalty minimization~\cite{wen2016trace},
and our \WTPM{} by both gradient descent (\WTPM{}-GD) and coordinate
descent (\WTPM{}-CD) methods. The numerical results for the case $p=5$ are
illustrated in Table~\ref{tab:ComparisonIters}. All these solvers use the
same initial matrix $X^{(0)}$ which is provided by Hartree--Fock theory,
and terminate when the error decreases to around $10^{-3}$. Let
\begin{equation}
    \label{eq:ErrorDefine}
    err^{(j)} = \max_{\ell = 1,2,\dotsc,p} \abs{\lambda_\ell - d_\ell^{(j)}},
\end{equation}
where $\lambda_\ell$, $d_\ell^{(j)}$ respectively denote the exact
eigenvalue and the corresponding Ritz value at $j$-th iteration. As for
the individual settings of each solver, in LOBPCG, the preconditioner is
not used. The penalty parameter in EigPen-B is set as 
\begin{equation}
    \mu = \max(r_p , 1).
\end{equation}
The weight matrix $W$ in \WTPM{} is set as the above statement
in~\eqref{eq:IntervalOfW} and $\mu = 1$.

Another thing we should emphasize is the computational complexity of each
solver in one iteration. In LOBPCG, EigPen-B and \WTPM{}-GD, the dominant
cost is $2p\cdot\mathrm{nnz}(A)$ flops to obtain matrix $AX$. But
\WTPM{}-CD updates each iteration mainly at the expense of
$2(p+2)\cdot\mathrm{nnz}(A_{:,\ell})$ flops, consisting of selecting
largest elements index, updating $Y$ and inadequately constructing $\nabla
\fmu$. That means, on average, the complexity of updating
$\tfrac{np}{p+2}$ times in \WTPM{}-CD equals that of updating once in the
other three solvers. Hence, we introduce ``\textbf{relative iteration}''
in \WTPM{}-CD, which means $\tfrac{np}{p+2}$ iterations and the actual
iteration number is shown in brackets in Table~\ref{tab:ComparisonIters}.

\begin{table}[htp]
    \caption{A comparison of updating iterations between eigensolvers for
    $p = 5$. The numbers in the ``\WTPM{}-CD'' row represent
    \textbf{relative iteration} outside the bracket and \textbf{actual
    iteration} in the bracket.}
    \label{tab:ComparisonIters}
    \centering
    \begin{tabular}{ccclcl}
        \toprule
        &  & \multicolumn{2}{c}{h2o} & \multicolumn{2}{c}{ham448} \\
        \cmidrule(lr){3-6} 
        &  & $err$ & iteration & $err$ & iteration \\
        \toprule
        LOBPCG &  & 8.676e-4 & 18 & 9.929e-4 & 68 \\
        EigPen-B &  & 2.971e-4 & 73 & 6.19e-4 & 134 \\
        WTPM-GD &  & 4.879e-4 & 185 & 3.68e-4 & 191 \\
        WTPM-CD &  & 3.901e-4 & \textbf{7} (283111) & 7.29e-4 & \textbf{4}
        (462000) \\
        \bottomrule
    \end{tabular}
\end{table}

The results in Table~\ref{tab:ComparisonIters} tell us that under the
requirement of $10^{-3}$ accuracy, \WTPM{}-CD shows the efficiency applied
to FCI matrices, since its theoretical complexity is much lower than the
others, though the coordinate method cannot take the advantage of level-3
BLAS operations as the others and its actual runtime is longer on small
systems. Instead, \WTPM{}-GD seems not an optimal choice due to the slow
convergence compared with EigPen-B and no compression of the coordinates.
That is because \WTPM{}'s theoretical condition number is larger than the
original trace-penalty minimization model if the elements of weight
matrix $W$ differ from each other.

Furthermore, Figure~\ref{fig:ErrorNnz} shows the convergence of the
eigenvalues in different eigensolvers and the number of nonzero elements
of the matrix $Y^{(j)}$ in \WTPM{}-CD varying against iteration. In LOBPCG
and \WTPM{}-CD, the error monotonically decreases to the specified
tolerance, but there exist spikes on the curve of error varying in
EigPen-B. That is because the BB stepsize cannot guarantee monotonicity
during minimization. The increasing tendency of $\mathrm{nnz}(Y^{(j)})$ in
Figure~\ref{fig:ErrorNnz} shows the effects of the coordinate descent
method and compression update. The $\mathrm{nnz}(Y^{(j)})$ is restricted
at a low level and so is the $\mathrm{nnz}(X^{(j)})$ because of the
inequality $\mathrm{nnz}(Y^{(j)}) \geqslant \mathrm{nnz}(X^{(j)})$ in
\WTPM{}-CD, which could significantly reduce the burden of the memory
source. This is what the other eigensolvers cannot do since the matrix
multiplication without compression will result in a dense matrix $AX$.

\begin{figure}[htp]
    \centering
    \subfloat[h2o]{\includegraphics[width=0.48\textwidth]{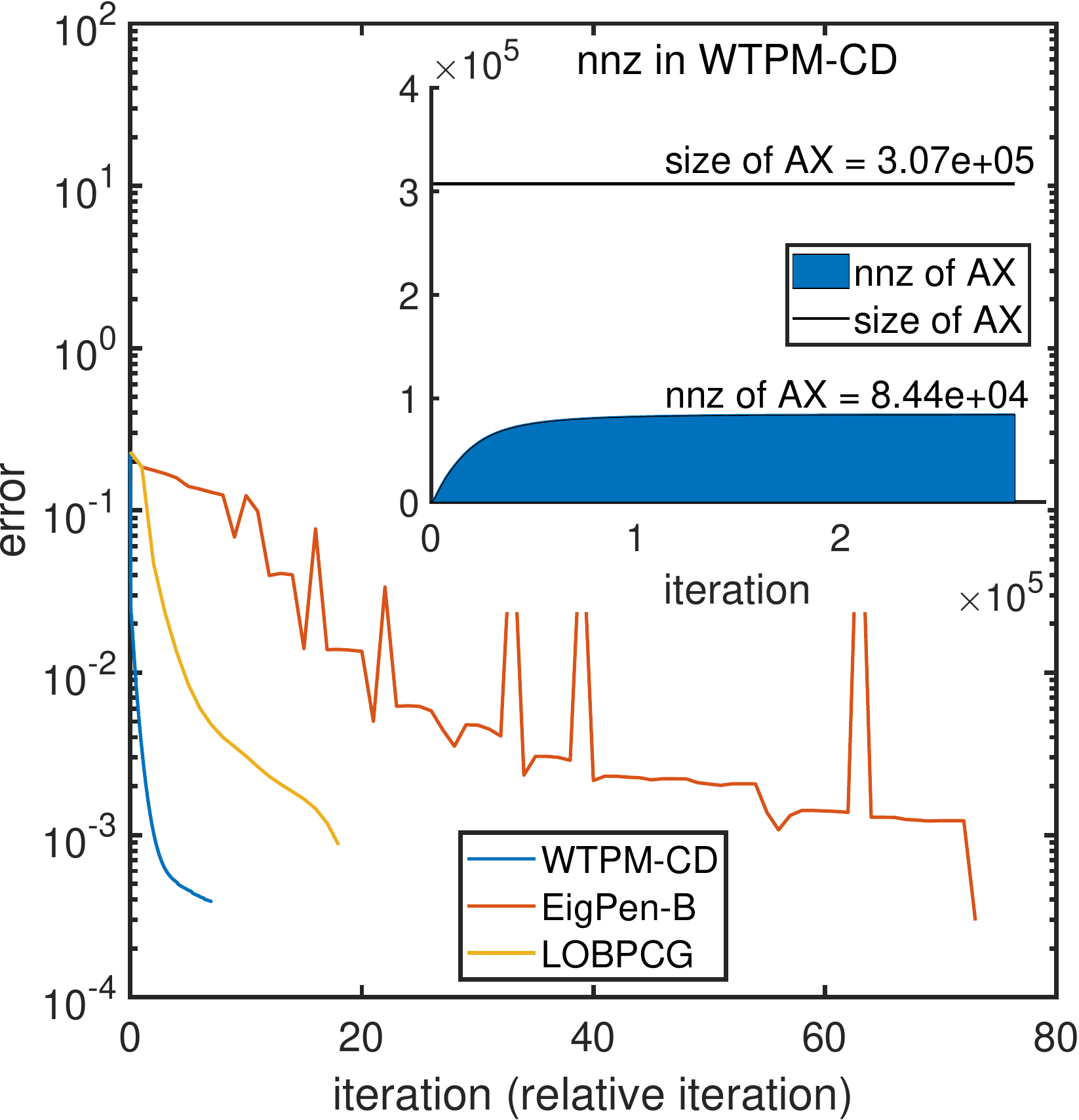}\label{fig:ErrorNnz:a}}
    \quad
    \subfloat[ham448]{\includegraphics[width=0.48\textwidth]{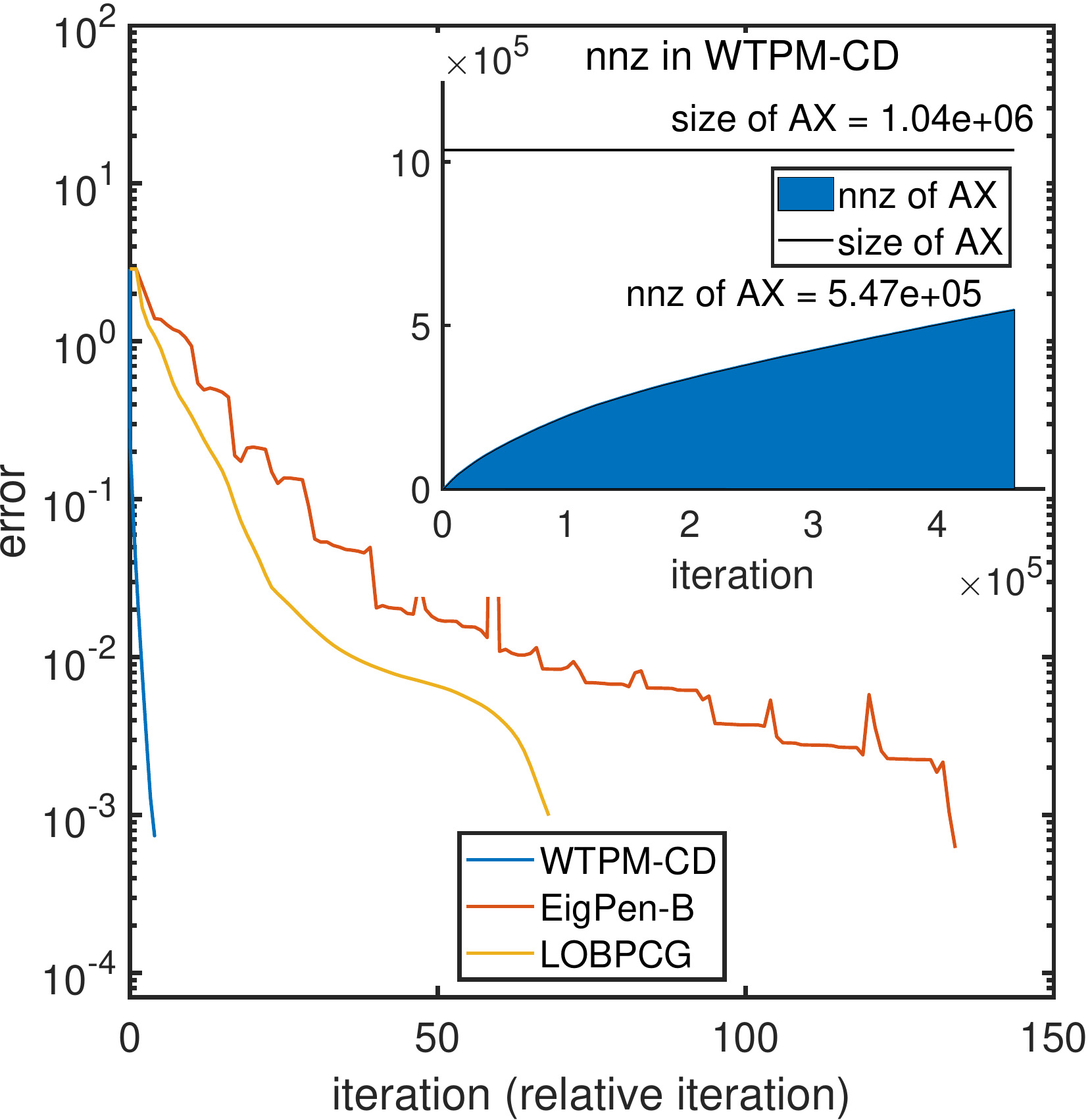}\label{fig:ErrorNnz:b}}
    \caption{Convergence of the smallest eigenvalues and number of nonzero
    elements (nnz) of the corresponding eigenvectors against iteration for
    $p=5$. The nnz of $AX$ in LOBPCG and EigPen-B are not shown here since
    it quickly increases to the maximum size of $AX$.}
    \label{fig:ErrorNnz}
\end{figure}

Besides, we apply \WTPM{}-CD with different weight matrices
$W$ to the ``h2o'' case to show the impact of $W$ on the convergence as
been analyzed in Section~\ref{sec:optimalW}. We use the exact eigenvalues
of $A$ to generate three different weight matrices. Let $w_1 =
\tfrac{\lambda_1 + \lambda_n}{2}$ and $w_p = \tfrac{\lambda_p +
\lambda_{p+1}}{2}$. We set
\begin{align*}
    \widetilde{W}:&\quad w_i = w_1 - (w_1 - w_p)\cdot \frac{i-1}{p-1}, \\
    \widehat{W}:&\quad w_i - w_{i+1} = \left(
        \sum_{j=1}^{p-1} (\lambda_{j+1}-\lambda_j)^{-1}
        \right)^{-1} \frac{w_1-w_p}{\lambda_{i+1}-\lambda_i}, \\
    \text{Random :}&\quad w_i \text{ is uniformly distributed in the interval } (w_p, w_1), 
\end{align*}
for $i = 2, 3, \dotsc, p-1$. Figure~\ref{fig:DiffWeight} shows the
convergence results for different weight matrices. It follows our
conclusion in Section~\ref{sec:optimalW} that $\widehat{W}$ is the best
choice among the three weight matrices and $\widetilde{W}$ makes the
convergence at least faster than the random choice does. In practice, if
$\widehat{W}$ cannot be obtained \textit{a priori} due to the cost of the
construction, $\widetilde{W}$ is often found efficient enough for
\WTPM{}-CD.

\begin{figure}[htp]
    \centering
    \includegraphics[width=0.5\textwidth]{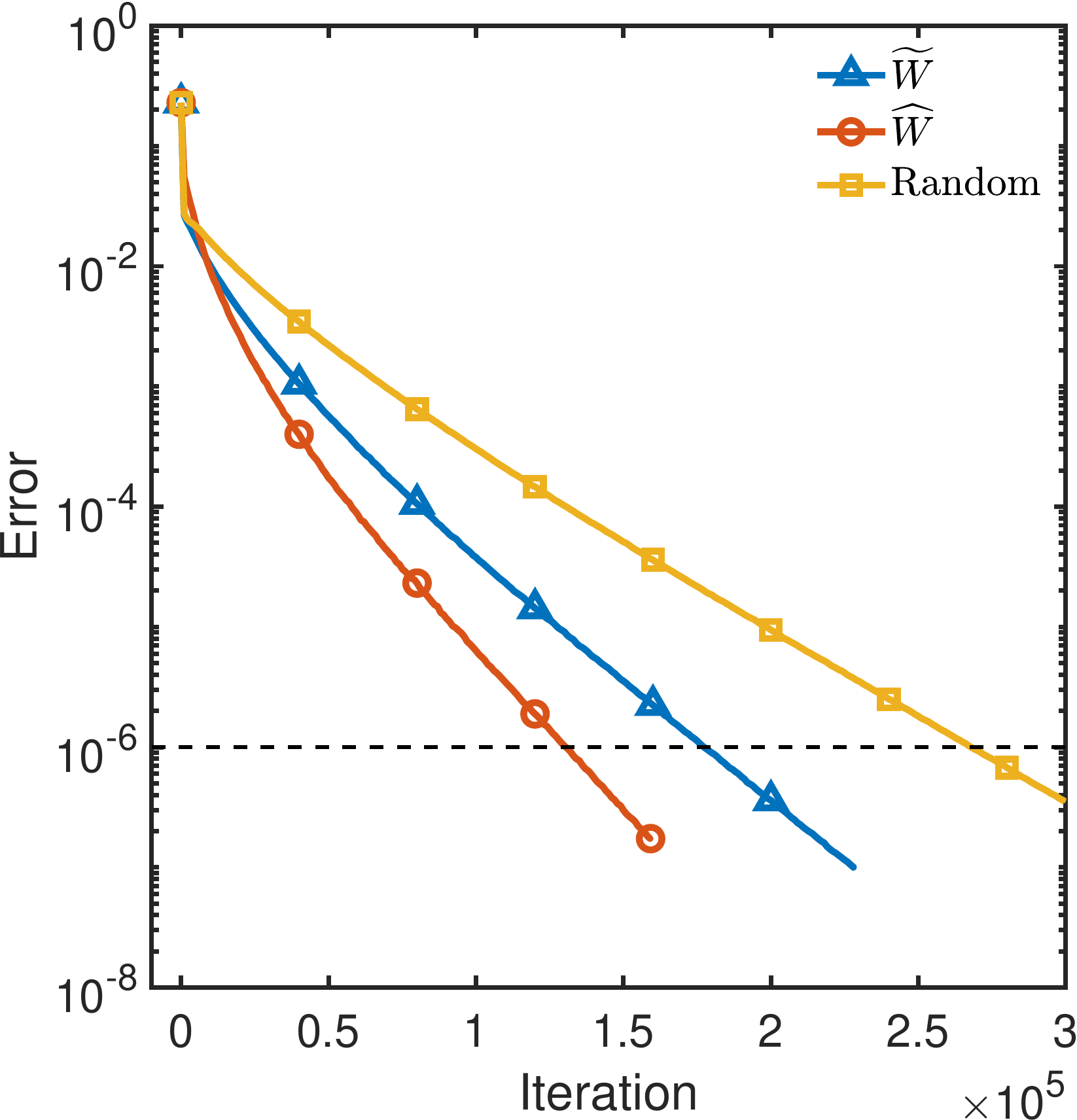}
    \caption{Convergence of the smallest eigenvalues against iteration for
    $p=5$ while \WTPM{}-CD applied to ``h2o'' matrix with different weight
    matrices.}
    \label{fig:DiffWeight}
\end{figure}

In summary, these examples indicate that \WTPM{}-CD could efficiently
solve the eigenvalue problem~\eqref{eq:eigenproblem} for FCI systems, and
provides an outstanding result for practical systems. In the next section,
we will illustrate the performance of \WTPM{}-CD in some larger FCI
matrices of more practical interest. 

\subsection{Performance in Large Systems}

This section provides some tests of larger FCI matrices by \WTPM{}-CD. All
programs are implemented in C++14 and compiled by Intel compiler 2021.5.0
with \texttt{-O3} option. MPI and OpenMP support is disabled for all
programs in this section. All of the tests in this section are produced on
a machine with Intel(R) Xeon(R) Gold 6226R CPU at 2.90 GHz and 1 TB
memory. The basic properties of matrices are illustrated in
Table~\ref{tab:LargeMatrix}. Two matrices correspond to \ch{H2O} and
\ch{C2} molecules generated via restricted Hartree--Fock (RHF) in PSI4
package \cite{Parrish_2017_Psi4OpensourceElectronic} with ccpVDZ atomic
orbital basis set.

\begin{table}[htp]
    \caption{Properties of Testing Molecule Systems.}
    \label{tab:LargeMatrix}
    \centering
    \begin{tabular}{cccc}
        \toprule
        Name & num of electrons & num of orbitals & matrix dimension \\
        \toprule
        h2o\_ccpvdz & 10 & 24 & $4.53\times 10^8$ \\ 
        c2\_ccpvdz & 12 & 28 & $1.77\times 10^{10}$ \\
        \bottomrule
    \end{tabular}
\end{table}

In the experiments, we use the $err^{(j)}$ defined
in~\eqref{eq:ErrorDefine} to measure the accuracy. Since we do not have
the exact energies of excited states, we use the energies obtained by our
algorithm without the compression threshold after a long enough time until
the first 7 digits to the right of the decimal remain unchanged as the
benchmark. The penalty parameter is simply set as $\mu = 1$, and the
weight matrix $W$ follows the rules in~\eqref{eq:IntervalOfW}. Thus we
obtain 
\begin{gather}
    W_{h2o} = \diag(-75.0, -75.175, -75.35), \label{eq:WH2O}\\
    W_{c2} = \diag(-75.138, -75.238, -75.338). \label{eq:WC2}
\end{gather}
And in the ``h2o\_ccpvdz'' case, the compression threshold is $1\times
10^{-6}$, and that in the ``c2\_ccpvdz'' case is $3\times 10^{-8}$.
Table~\ref{tab:ConvergenceLargeMatrix} and
Figure~\ref{fig:ConvergenceLargeMatrix} show the convergence of \WTPM{}-CD. 
The ``GS'', ``1st ES'' and ``2nd ES''
respectively denote the computed energies of the ground state, 1st excited
state and 2nd excited state. The ``time'' column denotes the time of the
computed energies first reaching the appointed precision, \ie, the runtime
when $err^{(j)} \leqslant tol$. This practical runtime shows the
efficiency of the \WTPM{}-CD on such molecules discretized by FCI.

\begin{table}[htp]
    \caption{Convergence of \WTPM{}-CD.}
    \label{tab:ConvergenceLargeMatrix}
    \centering
    \begin{tabular}{cccccccc}
        \toprule
        &   &   & \multicolumn{3}{c}{energy} &  &   \\
        \cmidrule(lr){4-6}
        Matrix & $p$ & $tol$ & GS & 1st ES & 2nd ES & time (s) &
        $\mathrm{nnz}(Y)$ \\
        \toprule
        h2o\_ccpvdz & 3 & 1.0e-2 & -76.24141  &  -75.88563  &  -75.86844 & 378 & 5.10e07 \\
         &  & 1.0e-3 & -76.24182 &  -75.89341 &  -75.86120 & 2179 & 5.16e07 \\
         &  & 1.0e-4 & -76.24186 &  -75.89430 &  -75.86050 & 8653 &
         5.16e07 \\
        \toprule
        c2\_ccpvdz & 3 & 1.0e-2 & -75.72888 & -75.63648 & -75.63609 & 560 & 3.03e08 \\
         &  & 1.0e-3 &  -75.73193  &  -75.64174  &   -75.63398   & 34248 & 1.26e09 \\
         &  & 1.0e-4 & -75.73196 & -75.64250 & -75.63327 & 102503 &
         1.27e09 \\
        \bottomrule
    \end{tabular}
\end{table}

\begin{figure}[htp]
    \centering
    \subfloat[h2o\_ccpvdz]{\includegraphics[width=0.48\textwidth]{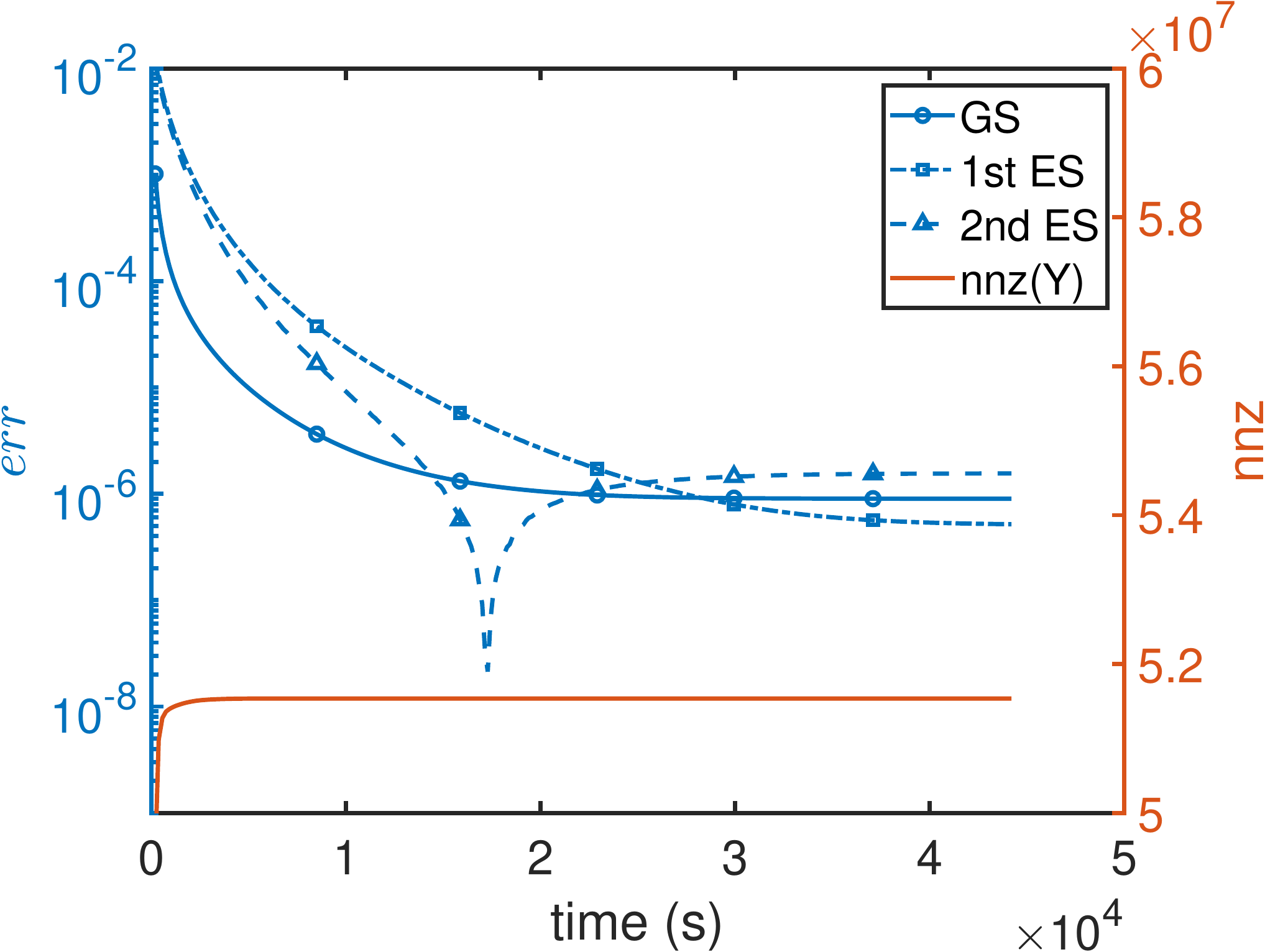}\label{fig:ConvergenceLargeMatrix:a}}
    \quad
    \subfloat[c2\_ccpvdz]{\includegraphics[width=0.48\textwidth]{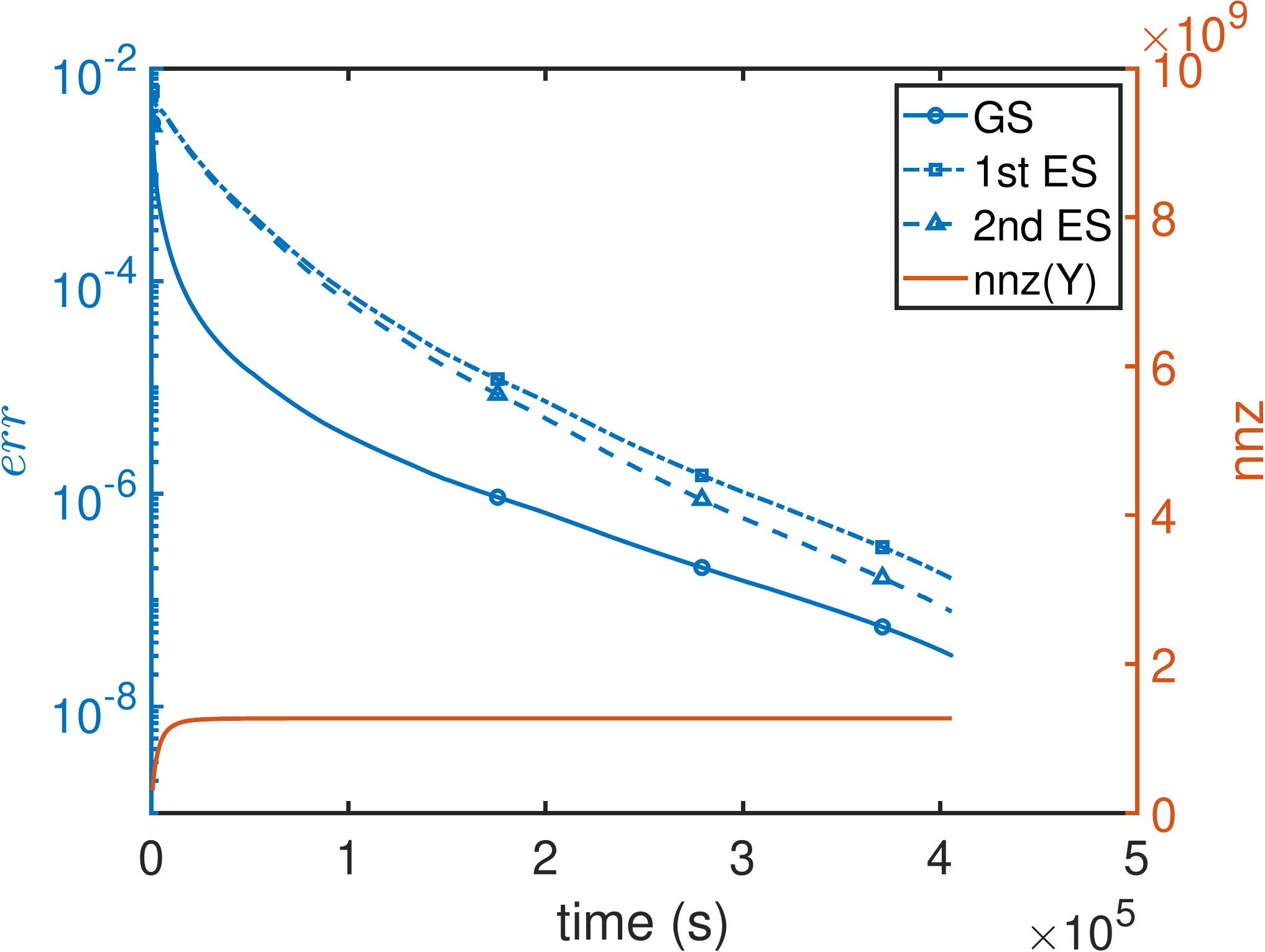}\label{fig:ConvergenceLargeMatrix:b}}
    \caption{Convergence of the smallest eigenvalues and number of nonzero elements (nnz) of the corresponding eigenvectors against runtime for $p=3$.}
    \label{fig:ConvergenceLargeMatrix}
\end{figure}

From Table~\ref{tab:LargeMatrix}, we can see that storing the dense
matrices $X$ (or $AX$) in double type needs at least 10 GB memory for
``h2o\_ccpvdz'' and 400 GB memory for ``c2\_ccpvdz''. This results in the
memory bottleneck of classical eigensolvers due to the unavoidable
gradient calculation or orthogonalization. The nnz$(Y)$ in
Table~\ref{tab:ConvergenceLargeMatrix} tells us that under the $10^{-4}$
accuracy, the cost of memory by \WTPM{}-CD decreases to 0.4 GB for
``h2o\_ccpvdz'' and 10 GB for ``c2\_ccpvdz'', which improves our
capability of solving larger FCI matrices generated by more complicated
particle systems. Figure~\ref{fig:ConvergenceLargeMatrix} shows the
details of the convergence, from which we can see the nnz$(Y)$ increases
quickly at the beginning and is fixed at a small number in comparison with
the dimension. Besides, the curve of ``2nd ES'' in
Figure~\ref{fig:ConvergenceLargeMatrix:a} has a different tendency from
the others. That is because the energy of the 2nd excited state in
``h2o\_ccpvdz'' is monotonically increasing during the iteration, and it
converges to a value larger than the benchmark. It indicates that if we
set a compression threshold, the converged result may lose some accuracy.

\section{Conclusion}
\label{sec:conclusion}

In this paper, we propose an eigensolver \WTPM{}-CD, which is an efficient
algorithm for FCI eigenvalue problems of quantum many-body systems. We
first propose a novel unconstrained minimization objective, namely
\WTPM{}, for Hermitian eigenvalue problems. The theoretical analysis of
the minimization model tells us that the global minimizers of \WTPM{} are
exactly the eigenvectors we expect instead of the invariant subspace,
so the orthogonalization process required by other methods is not needed
in \WTPM. Moreover, we calculate the exact condition number of the
Hessian operator, and use it to give a near-optimal weight matrix $W$. For
the algorithm framework, the coordinate descent method with compression
threshold reduces the number of nonzeros and the cost of storage. In
numerical experiments, compared with LOBPCG and EigPen-B solvers,
\WTPM{}-CD shows a better computational complexity on small systems. On
large-scale systems, \WTPM{}-CD guarantees its efficiency while the other
two solvers suffer from the bottleneck of memory.

There still exist some interesting works to be explored in the future.
First, we do not give theoretical proof on the convergence of the
\WTPM{}-CD algorithm. The performance of \WTPM{}-CD converging varies
while the picking rule changes. Some in-depth research on the convergence
theory is necessary to explain such a phenomenon. And some cheaper picking
rules, such as stochastic search, may be introduced to the coordinate
descent algorithm. The parallelization also attracts us to dive into it,
since at present we can only update one element in each iteration, and the
utilization efficiency of multicore is at a low level. We believe it is
possible to modify the coordinate descent method to update elements in a
batch and increase the parallel efficiency on shared memory systems.

\bibliographystyle{siamplain}
\bibliography{OEP}

\begin{thebibliography}{10}

\bibitem{absilOptimizationAlgorithmsMatrix2008}
{\sc P.-A. Absil, R.~Mahony, and R.~Sepulchre}, {\em Optimization Algorithms on
  Matrix Manifolds}, {Princeton University Press}, {Princeton}, 2008,
  \url{https://doi.org/10.1515/9781400830244}.

\bibitem{Baiardi_2017_VibrationalDensityMatrix}
{\sc A.~Baiardi, C.~J. Stein, V.~Barone, and M.~Reiher}, {\em Vibrational
  density matrix renormalization group}, Journal of Chemical Theory and
  Computation, 13 (2017), pp.~3764--3777,
  \url{https://doi.org/10.1021/acs.jctc.7b00329}.

\bibitem{barzilai1988two}
{\sc J.~Barzilai and J.~M. Borwein}, {\em Two-point step size gradient
  methods}, IMA Journal of Numerical Analysis, 8 (1988), pp.~141--148,
  \url{https://doi.org/10.1093/imanum/8.1.141}.

\bibitem{Berezin_1966_MethodSecondQuantization}
{\sc F.~A. Berezin}, {\em The Method of Second Quantization}, {Academic Press},
  {New York}, 1966.

\bibitem{bluntExcitedstateApproachFull2015}
{\sc N.~S. Blunt, S.~D. Smart, G.~H. Booth, and A.~Alavi}, {\em An
  excited-state approach within full configuration interaction quantum {{Monte
  Carlo}}}, The Journal of Chemical Physics, 143 (2015), p.~134117,
  \url{https://doi.org/10.1063/1.4932595}.

\bibitem{boothExactDescriptionElectronic2013}
{\sc G.~H. Booth, A.~Gr{\"u}neis, G.~Kresse, and A.~Alavi}, {\em Towards an
  exact description of electronic wavefunctions in real solids}, Nature, 493
  (2013), pp.~365--370, \url{https://doi.org/10.1038/nature11770}.

\bibitem{boothFermionMonteCarlo2009}
{\sc G.~H. Booth, A.~J.~W. Thom, and A.~Alavi}, {\em Fermion {{Monte Carlo}}
  without fixed nodes: {{A}} game of life, death, and annihilation in
  {{Slater}} determinant space}, The Journal of Chemical Physics, 131 (2009),
  p.~054106, \url{https://doi.org/10.1063/1.3193710}.

\bibitem{boyd_convex_2004}
{\sc S.~Boyd and L.~Vandenberghe}, {\em Convex Optimization}, Cambridge
  University Press, Cambridge, 2004.

\bibitem{chan2011thedensity}
{\sc G.~K.-L. Chan and S.~Sharma}, {\em The density matrix renormalization
  group in quantum chemistry}, Annual Review of Physical Chemistry, 62 (2011),
  pp.~465--481, \url{https://doi.org/10.1146/annurev-physchem-032210-103338}.

\bibitem{Chen2021b}
{\sc Z.~Chen, Y.~Li, and J.~Lu}, {\em On the global convergence of randomized
  coordinate gradient descent for non-convex optimization}, 2021,
  \url{https://doi.org/10.48550/ARXIV.2101.01323}.

\bibitem{clelandCommunicationsSurvivalFittest2010}
{\sc D.~Cleland, G.~H. Booth, and A.~Alavi}, {\em Communications: {{Survival}}
  of the fittest: {{Accelerating}} convergence in full
  configuration-interaction quantum {{Monte Carlo}}}, The Journal of Chemical
  Physics, 132 (2010), p.~041103, \url{https://doi.org/10.1063/1.3302277}.

\bibitem{Condon_1930_TheoryComplexSpectra}
{\sc E.~U. Condon}, {\em The theory of complex spectra}, Physical Review, 36
  (1930), pp.~1121--1133, \url{https://doi.org/10.1103/PhysRev.36.1121}.

\bibitem{corsettiOrbitalMinimizationMethod2014}
{\sc F.~Corsetti}, {\em The orbital minimization method for electronic
  structure calculations with finite-range atomic basis sets}, Computer Physics
  Communications, 185 (2014), pp.~873--883,
  \url{https://doi.org/10.1016/j.cpc.2013.12.008}.

\bibitem{gaoNewFirstOrderAlgorithmic2018}
{\sc B.~Gao, X.~Liu, X.~Chen, and Y.-x. Yuan}, {\em A new first-order
  algorithmic framework for optimization problems with orthogonality
  constraints}, SIAM Journal on Optimization, 28 (2018), pp.~302--332,
  \url{https://doi.org/10.1137/16M1098759}.

\bibitem{gaoTriangularizedOrthogonalizationfreeMethod2020}
{\sc W.~Gao, Y.~Li, and B.~Lu}, {\em Triangularized orthogonalization-free
  method for solving extreme eigenvalue problems}, 2020,
  \url{https://doi.org/10.48550/ARXIV.2005.12161}.

\bibitem{gaoGlobalConvergenceTriangularized2021}
{\sc W.~Gao, Y.~Li, and B.~Lu}, {\em Global convergence of triangularized
  orthogonalization-free method}, 2021,
  \url{https://doi.org/10.48550/ARXIV.2110.06212}.

\bibitem{Hardy_2001_Inequalities}
{\sc G.~H. Hardy, J.~E. Littlewood, and G.~P{\'o}lya}, {\em Inequalities},
  Cambridge Mathematical Library, {Cambridge Univ. Press}, {Cambridge},
  2nd~ed., 1988.

\bibitem{holmesHeatBathConfigurationInteraction2016}
{\sc A.~A. Holmes, N.~M. Tubman, and C.~J. Umrigar}, {\em Heat-bath
  configuration interaction: {{An}} efficient selected configuration
  interaction algorithm inspired by heat-bath sampling}, Journal of Chemical
  Theory and Computation, 12 (2016), pp.~3674--3680,
  \url{https://doi.org/10.1021/acs.jctc.6b00407}.

\bibitem{huBriefIntroductionManifold2020}
{\sc J.~Hu, X.~Liu, Z.-W. Wen, and Y.-X. Yuan}, {\em A brief introduction to
  manifold optimization}, Journal of the Operations Research Society of China,
  8 (2020), pp.~199--248, \url{https://doi.org/10.1007/s40305-020-00295-9}.

\bibitem{iterative1973huron}
{\sc B.~Huron, J.~P. Malrieu, and P.~Rancurel}, {\em Iterative perturbation
  calculations of ground and excited state energies from multiconfigurational
  zeroth-order wavefunctions}, The Journal of Chemical Physics, 58 (1973),
  pp.~5745--5759, \url{https://doi.org/10.1063/1.1679199}.

\bibitem{knowles1984new}
{\sc P.~J. Knowles and N.~C. Handy}, {\em A new determinant-based full
  configuration interaction method}, Chemical Physics Letters, 111 (1984),
  pp.~315--321, \url{https://doi.org/10.1016/0009-2614(84)85513-X}.

\bibitem{knowles1989determinant}
{\sc P.~J. Knowles and N.~C. Handy}, {\em A determinant based full
  configuration interaction program}, Computer Physics Communications, 54
  (1989), pp.~75--83, \url{https://doi.org/10.1016/0010-4655(89)90033-7}.

\bibitem{knowles1989unlimited}
{\sc P.~J. Knowles and N.~C. Handy}, {\em Unlimited full configuration
  interaction calculations}, The Journal of Chemical Physics, 91 (1989),
  pp.~2396--2398, \url{https://doi.org/10.1063/1.456997}.

\bibitem{knyazev2001toward}
{\sc A.~V. Knyazev}, {\em Toward the optimal preconditioned eigensolver:
  Locally optimal block preconditioned conjugate gradient method}, SIAM Journal
  on Scientific Computing, 23 (2001), pp.~517--541,
  \url{https://doi.org/10.1137/S1064827500366124}.

\bibitem{Knyazev_2007_BlockLocallyOptimal}
{\sc A.~V. Knyazev, M.~E. Argentati, I.~Lashuk, and E.~E. Ovtchinnikov}, {\em
  Block locally optimal preconditioned eigenvalue xolvers ({{BLOPEX}}) in
  {{Hypre}} and {{PETSc}}}, SIAM Journal on Scientific Computing, 29 (2007),
  pp.~2224--2239, \url{https://doi.org/10.1137/060661624}.

\bibitem{Lee2019}
{\sc J.~D. Lee, I.~Panageas, G.~Piliouras, M.~Simchowitz, M.~I. Jordan, and
  B.~Recht}, {\em First-order methods almost always avoid strict saddle
  points}, Mathematical Programming, 176 (2019), pp.~311--337,
  \url{https://doi.org/10.1007/s10107-019-01374-3}.

\bibitem{liCoordinateWiseDescentMethods2019}
{\sc Y.~Li, J.~Lu, and Z.~Wang}, {\em Coordinatewise descent methods for
  leading eigenvalue problem}, SIAM Journal on Scientific Computing, 41 (2019),
  pp.~A2681--A2716, \url{https://doi.org/10.1137/18M1202505}.

\bibitem{Lin_2016_ApproximatingSpectralDensities}
{\sc L.~Lin, Y.~Saad, and C.~Yang}, {\em Approximating spectral densities of
  large matrices}, SIAM Review, 58 (2016), pp.~34--65,
  \url{https://doi.org/10.1137/130934283}.

\bibitem{liuEfficientGaussNewton2015}
{\sc X.~Liu, Z.~Wen, and Y.~Zhang}, {\em An efficient
  {{Gauss}}\textendash{{Newton}} algorithm for symmetric low-rank product
  matrix approximations}, SIAM Journal on Optimization, 25 (2015),
  pp.~1571--1608, \url{https://doi.org/10.1137/140971464}.

\bibitem{luOrbitalMinimizationMethod2017}
{\sc J.~Lu and K.~Thicke}, {\em Orbital minimization method with {$\ell^1$}
  regularization}, Journal of Computational Physics, 336 (2017), pp.~87--103,
  \url{https://doi.org/10.1016/j.jcp.2017.02.005}.

\bibitem{luFullConfigurationInteraction2020}
{\sc J.~Lu and Z.~Wang}, {\em The full configuration interaction quantum monte
  carlo method through the lens of inexact power iteration}, SIAM Journal on
  Scientific Computing, 42 (2020), pp.~B1--B29,
  \url{https://doi.org/10.1137/18M1166626}.

\bibitem{nakanishi2019subspace}
{\sc K.~M. Nakanishi, K.~Mitarai, and K.~Fujii}, {\em Subspace-search
  variational quantum eigensolver for excited states}, Phys. Rev. Research, 1
  (2019), p.~033062, \url{https://doi.org/10.1103/PhysRevResearch.1.033062}.

\bibitem{olivares2015theabinitio}
{\sc R.~{Olivares-Amaya}, W.~Hu, N.~Nakatani, S.~Sharma, J.~Yang, and G.~K.-L.
  Chan}, {\em The ab-initio density matrix renormalization group in practice},
  The Journal of Chemical Physics, 142 (2015), p.~034102,
  \url{https://doi.org/10.1063/1.4905329}.

\bibitem{Parrish_2017_Psi4OpensourceElectronic}
{\sc R.~M. Parrish, L.~A. Burns, D.~G.~A. Smith, A.~C. Simmonett, A.~E.~I.
  DePrince, E.~G. Hohenstein, U.~Bozkaya, A.~Y. Sokolov, R.~Di~Remigio, R.~M.
  Richard, J.~F. Gonthier, A.~M. James, H.~R. McAlexander, A.~Kumar, M.~Saitow,
  X.~Wang, B.~P. Pritchard, P.~Verma, H.~F.~I. Schaefer, K.~Patkowski, R.~A.
  King, E.~F. Valeev, F.~A. Evangelista, J.~M. Turney, T.~D. Crawford, and
  C.~D. Sherrill}, {\em Psi4 1.1: {{An}} open-source electronic structure
  program emphasizing automation, advanced libraries, and interoperability},
  Journal of Chemical Theory and Computation, 13 (2017), pp.~3185--3197,
  \url{https://doi.org/10.1021/acs.jctc.7b00174}.

\bibitem{petruzieloSemistochasticProjectorMonte2012}
{\sc F.~R. Petruzielo, A.~A. Holmes, H.~J. Changlani, M.~P. Nightingale, and
  C.~J. Umrigar}, {\em Semistochastic projector {Monte Carlo} method}, Physical
  Review Letters, 109 (2012), p.~230201,
  \url{https://doi.org/10.1103/PhysRevLett.109.230201}.

\bibitem{sameh2000trace}
{\sc A.~Sameh and Z.~Tong}, {\em The trace minimization method for the
  symmetric generalized eigenvalue problem}, Journal of Computational and
  Applied Mathematics, 123 (2000), pp.~155--175,
  \url{https://doi.org/10.1016/S0377-0427(00)00391-5}.

\bibitem{schriberAdaptiveConfigurationInteraction2017}
{\sc J.~B. Schriber and F.~A. Evangelista}, {\em Adaptive configuration
  interaction for computing challenging electronic excited states with tunable
  accuracy}, Journal of Chemical Theory and Computation, 13 (2017),
  pp.~5354--5366, \url{https://doi.org/10.1021/acs.jctc.7b00725}.

\bibitem{sherrill1999configuration}
{\sc C.~D. Sherrill and H.~F. Schaefer}, {\em The configuration interaction
  method: Advances in highly correlated approaches}, Advances in Quantum
  Chemistry, 34 (1999), pp.~143--269,
  \url{https://doi.org/10.1016/S0065-3276(08)60532-8}.

\bibitem{Slater_1929_TheoryComplexSpectra}
{\sc J.~C. Slater}, {\em The theory of complex spectra}, Physical Review, 34
  (1929), pp.~1293--1322, \url{https://doi.org/10.1103/PhysRev.34.1293}.

\bibitem{Slater_1951_ASimplificationHF}
{\sc J.~C. Slater}, {\em A simplification of the hartree-fock method}, Phys.
  Rev., 81 (1951), pp.~385--390, \url{https://doi.org/10.1103/PhysRev.81.385}.

\bibitem{tubmanDeterministicAlternativeFull2016}
{\sc N.~M. Tubman, J.~Lee, T.~Y. Takeshita, M.~{Head-Gordon}, and K.~B.
  Whaley}, {\em A deterministic alternative to the full configuration
  interaction quantum {{Monte Carlo}} method}, The Journal of Chemical Physics,
  145 (2016), p.~044112, \url{https://doi.org/10.1063/1.4955109}.

\bibitem{vecharynskiProjectedPreconditionedConjugate2015}
{\sc E.~Vecharynski, C.~Yang, and J.~E. Pask}, {\em A projected preconditioned
  conjugate gradient algorithm for computing many extreme eigenpairs of a
  {{Hermitian}} matrix}, Journal of Computational Physics, 290 (2015),
  pp.~73--89, \url{https://doi.org/10.1016/j.jcp.2015.02.030}.

\bibitem{wangCoordinateDescentFull2019a}
{\sc Z.~Wang, Y.~Li, and J.~Lu}, {\em Coordinate descent full configuration
  interaction}, Journal of Chemical Theory and Computation, 15 (2019),
  pp.~3558--3569, \url{https://doi.org/10.1021/acs.jctc.9b00138}.

\bibitem{Li_2022_CDFCIforExcitedStates}
{\sc Z.~Wang, Z.~Zhang, J.~Lu, and Y.~Li}, {\em Coordinate descent full
  configuration interaction for excited states}, 2023,
  \url{https://doi.org/10.48550/ARXIV.2304.13380}.

\bibitem{wen2016trace}
{\sc Z.~Wen, C.~Yang, X.~Liu, and Y.~Zhang}, {\em Trace-penalty minimization
  for large-scale eigenspace computation}, Journal of Scientific Computing, 66
  (2016), pp.~1175--1203, \url{https://doi.org/10.1007/s10915-015-0061-0}.

\bibitem{white1999ab}
{\sc S.~R. White and R.~L. Martin}, {\em Ab initio quantum chemistry using the
  density matrix renormalization group}, The Journal of Chemical Physics, 110
  (1999), pp.~4127--4130, \url{https://doi.org/10.1063/1.478295}.

\bibitem{zhouChebyshevDavidsonAlgorithm2007}
{\sc Y.~Zhou and Y.~Saad}, {\em A {{Chebyshev}}\textendash{{Davidson}}
  algorithm for large symmetric eigenproblems}, SIAM Journal on Matrix Analysis
  and Applications, 29 (2007), pp.~954--971,
  \url{https://doi.org/10.1137/050630404}.

\bibitem{zhou2008block}
{\sc Y.~Zhou and Y.~Saad}, {\em Block {Krylov--Schur} method for large
  symmetric eigenvalue problems}, Numerical Algorithms, 47 (2008),
  pp.~341--359, \url{https://doi.org/10.1007/s11075-008-9192-9}.

\end{thebibliography}
\end{document}